\documentclass[10pt,twoside, english, reqno]{amsart}
\usepackage{graphicx}
\usepackage{a4wide}
\usepackage{xcolor}
\usepackage{marvosym}
\usepackage{wasysym}
\usepackage{bbm}
\usepackage{todonotes}
\usepackage{float}
\usepackage{enumerate}
\usepackage{enumitem}
\usepackage{units}
\usepackage{esint}
\usepackage[colorlinks]{hyperref}
\usepackage{upref}
\usepackage{xifthen}
\usepackage{mathtools}
\usepackage{amssymb}

\graphicspath{{figures/}}

\newtheoremstyle{note}% name
  {3pt}%      Space above
  {3pt}%      Space below
  {}%         Body font
  {}%         Indent amount (empty = no indent, \parindent = para indent)
  {\itshape}% Thm head font
  {:}%        Punctuation after thm head
  {.5em}%     Space after thm head: " " = normal interword space;
  {}%         Thm head spec (can be left empty, meaning `normal')

\newtheorem{lemma}{Lemma}[section]
\newtheorem{proposition}[lemma]{Proposition}

\newtheorem{theorem}[lemma]{Theorem}
\newtheorem{remark}[lemma]{Remark}

\newtheorem*{maintheorem*}{Main Theorem}
\theoremstyle{definition}{\newtheorem{definition}[lemma]{Definition}}

\theoremstyle{note}{
\newtheorem*{claim*}{Claim}}

\allowdisplaybreaks
\numberwithin{equation}{section}

\newcommand{\eps}{\varepsilon}
\newcommand{\sol}{\mathcal{U}}
\newcommand{\sole}{\mathcal{U}_\tau}
\newcommand{\te}{\theta_\eps}
\newcommand{\dom}{\Omega}
\newcommand{\projU}{\mathbb{P}_N}
\newcommand{\projW}{\mathbb{Q}_N}

\newcommand{\bU}{\mathbf{U}}
\newcommand{\bW}{\mathbf{W}}
\newcommand{\bH}{\mathbf{H}}

\newcommand{\bM}{\mathbf{M}}

\newcommand{\bue}{\mathbf{u}_{\tau}}
\newcommand{\bme}{\mathbf{m}_{\tau}}
\newcommand{\bwe}{\mathbf{w}_{\tau}}
\newcommand{\bhe}{\mathbf{h}_{\tau}}

\newcommand{\vfe}{\varphi_{\tau}}

\newcommand{\norm}[1]{\left\| #1 \right\|}
\newcommand{\weakstar}{\overset{\star}\rightharpoonup}

\renewcommand{\i}{\ifmmode\mathit{\mathchar"7010 }\else\char"10 \fi}
\renewcommand{\j}{\ifmmode\mathit{\mathchar"7011 }\else\char"11 \fi}
\newcommand{\R}{\mathbb{R}}
\newcommand{\N}{\mathbb{N}}

\newcommand{\Grad}{\nabla}
\newcommand{\Div}{\operatorname{div}}
\newcommand{\Curl}{\operatorname{curl}}

\newcommand{\hdiv}{L^2_{\Div}(\dom)}
\newcommand{\V}{H^1_{\Div}(\dom)}

\newcommand{\bu}{\mathbf{u}}
\newcommand{\bm}{\mathbf{m}}
\newcommand{\bw}{\mathbf{w}}
\newcommand{\bh}{\mathbf{h}}
\newcommand{\bb}{\mathbf{b}}

\newcommand{\bus}{\mathbf{u}^\sigma}
\newcommand{\bms}{\mathbf{m}^\sigma}
\newcommand{\bws}{\mathbf{w}^\sigma}
\newcommand{\bhs}{\mathbf{h}^\sigma}

\newcommand{\phis}{\varphi^\sigma}
\newcommand{\busn}{\mathbf{u}^{\sigma_n}}
\newcommand{\bmsn}{\mathbf{m}^{\sigma_n}}
\newcommand{\bwsn}{\mathbf{w}^{\sigma_n}}
\newcommand{\bhsn}{\mathbf{h}^{\sigma_n}}

\newcommand{\phisn}{\varphi^{\sigma_n}}
\newcommand{\normal}{\mathbf{n}}
\newcommand{\ajn}{\alpha_{j,N}}
\newcommand{\bjn}{\beta_{j,N}}
\newcommand{\bUN}{\bU_N}
\newcommand{\bWN}{\bW_N}

\newcommand{\lint}{\int_0^T\!\!\!\int_{\dom}\!}
\newcommand{\mi}{m^{(i)}}
\newcommand{\hi}{h^{(i)}}
\newcommand{\bik}{b^{(i)}_K}
\newcommand{\bikd}{b^{(i)}_{K,\delta}}

{%

\begin{enumerate}}%
{\end{enumerate}}

{%

\begin{enumerate}}%
{\end{enumerate}}

\begin{document}
\title[Global weak solutions of the Rosensweig system] {On the dynamics of ferrofluids: Global weak solutions to the Rosensweig system and rigorous convergence to equilibrium}	
\author[R.H. Nochetto]{Ricardo H. Nochetto}
\address[R.H. Nochetto]{\newline
	Department of Mathematics \\ University of Maryland \\ College Park, MD 20742, USA.}
\author[K. Trivisa]{Konstantina Trivisa}
\address[K. Trivisa]{\newline
 Department of Mathematics \\ University of Maryland \\ College Park, MD 20742, USA.}
\email[]{\href{trivisa@math.umd.edu}{trivisa@math.umd.edu}}
\urladdr{\href{http://www.math.umd.edu/~trivisa}{math.umd.edu/\~{}trivisa}}

\author[F. Weber]{Franziska Weber}
\address[F. Weber]{\newline
	Department of Mathematical Sciences \\ Carnegie Mellon University \\ Pittsburgh, PA 15213, USA}
\email[]{\href{franzisw@andrew.cmu.edu}{franzisw@andrew.cmu.edu}}

\date{\today}



\begin{abstract}
This article establishes the global existence of weak solutions to a model proposed by  Rosensweig \cite{Rosensweig1985} for the dynamics of ferrofluids. The system is expressed by the conservation of linear momentum, the incompressibility condition, the conservation of angular momentum, and the evolution of the magnetization.  
The existence proof is inspired by the DiPerna-Lions theory of renormalized solutions.
In addition,  the  rigorous relaxation limit of the equations of ferrohydrodynamics towards the quasi-equilibrium is investigated. The proof relies on the relative entropy method, which involves constructing  a suitable functional, analyzing  its time evolution and obtaining convergence results for the sequence of approximating solutions. 
\end{abstract}

\maketitle

\section{Introduction}\label{S1}
\subsection{Motivation}
Ferrofluids are stable colloidal dispersions of nano-sized particles of ferro- or ferrimagnetic particles in a carrier liquid. A crucial property of ferrofluids is that they can be actuated by magnetic fields, upon changing the position and strength of magnets, and be forced to flow to precise locations or display a specific patterns. These complex fluids have an array of engineering and biomedical applications. Ferrofluids have the capability of reducing friction, making them useful in a variety of electronic and transportation applications. They can also be used as a liquid seal in many electronic devices, for instance in computer hard-drives where they can be utilized to form a seal around the rotating shaft~\cite{Raj1995,Scherer2005} or in loudspeakers for cooling and damping unwanted resonances~\cite{Miwa2003}. On the other hand, ferrofluids have been instrumental in transporting medications to exact locations within the human body (drug delivery), they have been of use as contrasting agents for Magnetic Resonance Imaging (MRI) scans. More recently, ferrofluids have been of use in on-going research investigations aiming at  the creation of  an artificial heart; by surrounding the heart with magnets, the ferrofluid fixed to frame of the heart will expand and contract when needed, imitating the pumping of the real organ. We refer the reader to  \cite{Pankhurst2003} and \cite{Wang2008} for an overview of relevant biomedical applications. 

Although our understanding of the dynamics of ferrofluids has evolved in recent years, many aspects of ferrohydrodynamics remain largely unexplored, especially experimentally. 
This article is part of a research program which aims at enhancing our understanding of the properties and dynamics of ferrofluids through the analysis of models that are relevant to practical applications.

\subsection{Governing Equations}
The goal of this work is the rigorous investigation of the solvability of the equations proposed by Rosensweig~\cite{Rosensweig1985} that describe
the flow of an incompressible ferrofluid subjected to an external magnetic field. 
In this model, the dynamics of the linear velocity $\bu$, the pressure $p$, the angular momentum $\bw$ and the magnetization $\bm$  on a bounded simply connected domain $\dom\subset\R^d$, $d=2,3$, are governed by the conservation of linear momentum, the incompressibility condition, the conservation of angular momentum, the transport of the magnetization and the magnetostatic equation for the magnetic field $\bh$ as follows (cf. \cite{Rosensweig1985, Rosensweig1985a, Odenbach2002}):
\begin{subequations}
	\label{eq:rosenzweig}
	\begin{align}
	\bu_t +(\bu\cdot\Grad) \bu- (\nu+\nu_r)\Delta \bu+\Grad p &= 2 \nu_r \Curl \bw +\mu_0 (\bm\cdot\Grad)\bh,\label{seq:dtu}\\
	\Div \bu &=0\label{seq:divu},\\
	\bw_t + (\bu\cdot \Grad)\bw - c_1\Delta \bw -c_2 \Grad\Div\bw +4 \nu_r \bw &= 2\nu_r \Curl \bu+\mu_0 \bm\times \bh,\label{seq:dtw}\\
	\bm_t +(\bu\cdot \Grad)\bm &=\bw\times \bm-\frac{1}{\tau}(\bm-\kappa_0 \bh),\label{seq:dtm}\\
	\Curl \bh &=0,\label{seq:curlh}\\
	\Div (\bh + \bm)&=0\label{seq:divb},
	\end{align}
\end{subequations}
with suitable boundary conditions discussed later.
The forcing term $${\bf F}= \mu_0 (\bm\cdot\Grad)\bh$$ in the linear momentum equation is the so-called Kelvin force. The term $\mu_0 \bm$ represents the vector moment per unit volume.
The {\em effective magnetizing field} $\bh$  is given by
\begin{equation} 
{\bh} = \bh_{a} + \bh_d, \label{effmf}
\end{equation}
where $\bh_a$ is the so-called {\em applied magnetic field} and $\bh_d$ is the {\em demagnetizing field}. The former is assumed to be smooth and both rotation and divergence free in $\mathbb{R}^3$. The latter is rotation free in $\mathbb{R}^3$. Equation \eqref{seq:divb} is the Maxwell equation $\Div \bb = 0$ in $\mathbb{R}^3$ for the magnetic induction given by $\bb=\mu_0(\bh+\bm)$ in $\Omega$ and $\bb=\mu_0\bh$ outside $\Omega$ where the magnetization $\bm$ vanishes. Invoking the customary, although somewhat arbitrary, assumption that $\bh_d = 0$ outside $\Omega$ along with the continuity of the normal components of $\bb$ and $\bh_a$ across the boundary $\Gamma$ of $\Omega$, we deduce
\begin{equation}\label{normal-hd}
\bh_d\cdot\normal = - \bm\cdot\normal\quad\textrm{on }\Gamma,
\end{equation}
where $\normal$ is the unit outer normal on $\Gamma$.
Whenever the magnetization $\bm$ is small relative to $\bh_a$, the demagnetizing field $\bh_d$ is negligible and the effective magnetic field satisfies $\bh \approx \bh_a.$ If $\bh$ is a given field rather than the solution of the magnetostatics equations~\eqref{seq:curlh}--\eqref{seq:divb}), then the analysis of the reduced system~\eqref{seq:dtu}--\eqref{seq:dtm} simplifies considerably. However, recent numerical simulations for a related two-phase flow model~\cite{Nochetto2016b} indicate that the reduced system may not be able to capture the whole physical behavior of ferrofluids. The famous Rosensweig instability, for example, can only be reproduced when $\bh_d$ is present or equivalently when $\bh$ solves the magnetostatics equation~\eqref{seq:curlh}--\eqref{seq:divb} (see~\cite[Figures 6 and 7]{Nochetto2016b}). We will therefore focus here on the analysis of the full system~\eqref{eq:rosenzweig}. 

The model~\eqref{eq:rosenzweig} is derived under the following physically grounded, yet restrictive, hypotheses:
\begin{itemize}
	\item [(A1)] The ferromagnetic particles are spherical.
	\item [(A2)] The ferrofluid is a monodisperse mixture, in the sense  that the ferromagnetic particles are of the same mass/size.
	\item [(A3)] The density of ferromagnetic particles (number of particles per unit volume) in the carrier liquid is considered to be homogeneous.
	\item [(A4)] No agglomeration, clumping, anisotropic behavior (e.g. formation of chains), nor particle-to-particle interactions are considered.
	\item [(A5)] The induced fields ($\bm$ and $\bh_d$) are unable to perturb the applied magnetic field $\bh_a.$
\end{itemize}
Even though these assumptions might restrict the applicability of the Rosensweig model, there is a large class of physical situations in which they apply (cf. \cite{Rosensweig1985}). 
The derivation follows the strategy that is common in the theory of thermodynamics, namely we start stating fundamental principles such as the conservation of linear and angular momentum, the conservation of mass and the evolution of the magnetization   in the presence of stress tensors and other quantities  which satisfy rather general constitutive laws. Next, one  proceeds by writing the Clausius-Duhem inequality and with the  calculation the entropy production rate, which according to the second law of thermodynamics is assumed to be a nonnegative measure. This requirement  imposes additional restrictions on various quantities (tensors, forces  and parameters) in the system, which result to the form of the nonlinear system stated above. In particular, the material constants $\nu,\nu_r, \mu_0, j, c_{\alpha}, c_d, c_0, \kappa_0$ are assumed to be  nonnegative and are chosen in such a way so that  the Clausius-Duhem inequality is satisfied (cf. \cite{Rosensweig1985, Odenbach2002,Lukaszewicz1999}). 

We refer the reader to Rinaldi and Zahn~\cite{RinaldiZahn2002}, Sunil, Chand and Bharti~\cite{Sunil2007}, Zahn and Greer~\cite{Zahn1995} for further remarks.

\medskip\noindent
{\it Initial and Boundary Conditions.}
We assume that the initial data is such that
\begin{equation}\label{eq:init}
\bu(0,x)=\bu_0(x)\in \hdiv, \quad \bw(0,x)=\bw_0(x)\in L^2 (\dom),\quad \bm(0,x)=\bm_0(x)\in L^2(\dom). 
\end{equation}
The external applied field $\bh_a$ is assumed to be smooth in space and time and divergence free.
Moreover, we use the boundary conditions on $\Gamma:=\partial\dom$
\begin{equation}
\label{eq:bc1}
\left.\bu(t,\cdot)\right|_\Gamma= 0,\quad \left.\bw(t,\cdot)\right|_\Gamma = 0,\quad t\in [0,T]. 
\end{equation}
Under the assumption that the domain $\dom$ is simply-connected, we can write $\bh$ as the gradient of a potential $\bh=\Grad\varphi$, by \eqref{seq:curlh}, whence \eqref{seq:divb} becomes an elliptic equation for the potential $\varphi$
\begin{equation}
\label{eq:potential1}
-\Delta\varphi = \Div \bm.
\end{equation}
In view of \eqref{effmf} and \eqref{normal-hd}, we infer $\bh\cdot\normal=(\bh_a-\bm)\cdot\normal$ and a Neumann boundary condition for $\varphi$
\begin{equation}\label{eq:potentialbc}
\frac{\partial\varphi}{\partial n}=(\bh_a-\bm)\cdot\normal\quad \textrm{on }[0,T]\times\Gamma\quad\text{ with } \int_{\dom}\varphi(t,x) dx = 0.
\end{equation}
We refer to~\cite[Section 2.2]{Nochetto2016} for a physically motivated discussion of boundary conditions. We observe that $\bm$ dictates $\bh_d$. Since $\Curl \bh_d=0$ there exists a scalar potential $\zeta$ such that $\nabla\zeta=\bh_d$ in $\Omega$. Invoking $\Div(\bh+\bm)=\Div(\bh_d+\bm)=0$ along with \eqref{normal-hd}, we deduce
\[
-\Delta\zeta = \Div\bm \quad\textrm{in }\Omega,
\qquad
\frac{\partial \zeta}{\partial n} = - \bm\cdot\normal \quad\textrm{on }\Gamma.
\]
This determines a unique function $\zeta$ with vanishing mean-value, hence the crucial field $\bh_d$.

Existence of global weak solutions to the system \eqref{seq:dtu}-\eqref{seq:divb} in the presence of additional diffusion $\sigma \Delta \bm$, $\sigma>0$, in \eqref{seq:curlh}  has been shown by Amirat et al.~\cite{Amirat2008}.
Results  on the  local in  time existence of strong solutions has been shown~\cite{Amirat2010}.  To the best of our knowledge our article is the first that establishes the global existence of weak solutions in the absence of additional diffusion in the magnetization equation. In particular, we prove the following existence result,
\begin{theorem}[Global Existence]
	Assume the initial data $(\bu_0,\bw_0,\bm_0)$ satisfies~\eqref{eq:init}, the effective magnetizing field $\bh$ is given by~\eqref{effmf} and the applied magnetizing field $\bh_a$ is smooth and divergence free. 
	Then the problem \eqref{seq:dtu}--\eqref{seq:divb} with boundary conditions \eqref{eq:bc1} and~\eqref{eq:potentialbc} has a global weak solution $\sol:=(\bu,\bw,\bm,\bh)$, as in Definition~\ref{def:weaksol}, satisfying the energy inequality
	\begin{equation}\label{eq:energy0}
	\int_{\dom} E(\sol)(t) dx +\int_0^t\int_{\dom} D(\sol)(s)\,dx ds\leq \int_{\dom} E(\sol)(0) dx +\mu_0\int_0^t\int_{\dom}\partial_t\bh_a\cdot \bh\, dx ds,
	%		 -\mu_0\int_0^t\int_{\dom} f_s\varphi dx ds,
	\end{equation}
	where the energy $E$ is defined by
	\begin{equation}\label{eq:defE}
	E(\sol) = \frac{1}{2}\left(|\bu|^2 +|\bw|^2+ \frac{\mu_0}{\kappa_0}|\bm|^2+\mu_0|\bh|^2\right),
	\end{equation}
	and the dissipation functional $D$ is defined by
	\begin{equation}\label{eq:defD}
	D(\sol) = \left(\nu |\Grad \bu|^2 + c_1 |\Grad\bw|^2 + c_2|\Div \bw|^2+\nu_r|\Curl \bu -2\bw|^2+\frac{\mu_0}{\tau \kappa_0}|\bm-\kappa_0\bh|^2 \right).
	\end{equation}
\end{theorem}
\begin{remark}
	Since by Young's inequality
	\begin{equation*}
	\left|\mu_0\int_0^t\int_{\dom}\partial_t\bh_a\cdot \bh\, dx ds\right|\leq \frac{\mu_0}{2}\int_0^t\int_{\dom}|\partial_t\bh_a|^2\, dx ds+\frac{\mu_0}{2}\int_0^t\int_{\dom}|\bh|^2\, dx ds,
	\end{equation*}
	and $\bh_a$ is smooth, we combine the energy inequality~\eqref{eq:energy0} with Gr\"{o}nwall's inequality to obtain that
	\begin{equation*}
		\bu,\bw\in L^\infty(0,T;L^2(\dom))\cap L^2(0,T;H^1(\dom)),\quad 
	\bm, \bh \in L^\infty(0,T;L^2(\dom)),
	\end{equation*}
	and hence the term $\mu_0\int_0^t\int_{\dom}\partial_t\bh_a\cdot \bh\, dx ds$ on the right hand side of the energy inequality~\eqref{eq:energy0} is bounded.
\end{remark}
\begin{remark}
	The technique used to prove this result can be extended in a straightorward fashion to prove existence of global weak solutions of the diffusive interface model for two-phase ferrofluid flows that was introduced in~\cite{Nochetto2016b}.
\end{remark}

This theorem is established by constructing  a suitable sequence of approximating solutions inspired by the result of Amirat et al.~\cite{Amirat2008}. 
The lack of sufficient regularity in the sequence of approximate solutions, as obtained employing the energy estimate, presents a big challenge in the analysis. 
 More specifically, the main energy bound  only yields $\bm\in L^\infty(0,T;L^2(\dom))$ uniformly for the approximating sequence of solutions. The reader should contrast this space setting  to the result in~\cite{Amirat2008}, where the additional artificial dissipation term in the magnetization equation yields $H^1$-regularity in space. 

Another major difficulty arises due to the presence of the a priori unbounded Kelvin force term $\mu_0(\bm\cdot\Grad)\bh$  in the momentum equation. This challenge is addressed by deriving a new identity (see Lemma~\ref{lem:integrableweaksol}) that allows us to define the Kelvin force term in a distributional sense, and  the definition of weak solutions of~\eqref{eq:rosenzweig} is modified  accordingly in light of the new formulation. 

In addition, passing to the limit in the approximations requires us to establish compactness in $L^2$ for  the approximating sequence of the magnetization $\bm.$ This is achieved by proving that the magnetization is `renormalized' in the spirit of DiPerna-Lions theory for compressible fluids~\cite{Lions1996,Lions1998book}. Combining this new feature of the equation  with an additional identity obtained  for the weak limits of the magnetization and the magnetic field from the magnetostatics equation assists in establishing the result.

The long term goal of this research effort is the construction of convergent numerical schemes for the approximation of this nonlinear system.

\subsection{Relaxation time and quasi-equilibrium}
In practical applications, the parameter $\tau>0$ in equation~\eqref{seq:dtm}, the relaxation time, is often very small.
According to~\cite{Shliomis2002}, it is given by
\begin{equation*}
\frac{1}{\tau}=\frac{1}{\tau_B}+\frac{1}{\tau_N}, \quad \text{or}\quad \tau =\frac{\tau_B\tau_N}{\tau_N+\tau_B},
\end{equation*}
where $\tau_B$ is the Brownian time and $\tau_N$ is the N\'{e}el time. In Brownian relaxation, the magnetization vector $\bm$ rotates synchronically with the particle, while in the N\'{e}el relaxation mechanism, the magnetization vector $\bm$ rotates inside the particle and the particle itself does not rotate. Depending on the particle size, one or the other mechanism dominates. For particles with a smaller diameter than the so-called Shliomis' diameter $d_S$, the N\'{e}el time satisfies $\tau_N\ll \tau_B$ and hence $\tau\approx \tau_N$. For bigger particles, $\tau_B\ll\tau_N$ and hence $\tau\approx\tau_B$. The N\'{e}el time can be very small: For example for particles with diameter $d\approx (0.3 - 0.5) d_S$ it is of order $\tau\approx\tau_N\approx 10^{-9}s$ and hence the magnetization vector $\bm$ becomes parallel to the magnetic field vector $\bh$ almost immediately. For the opposite case, $\tau\approx \tau_B\approx 10^{-5} - 10^{-4} s$ and hence the magnetization vector does not need to be parallel to $\bh$~\cite{Shliomis2002}. Rinaldi~\cite[page 54]{Rinaldi2002} states that the relaxation time may be of the order $\tau\approx 10^{-7} - 10^{-5}s$.
One may therefore assume that it is of interest to investigate the behavior solutions of~\eqref{eq:rosenzweig} as $\tau\to 0$. Formally, setting $\tau=0$ in~\eqref{eq:rosenzweig}, we obtain the following system:
\begin{subequations}
	\label{eq:limitsys}
	\begin{align}
	\bU_t +(\bU\cdot\Grad) \bU- (\nu+\nu_r)\Delta \bU+\Grad P &= 2 \nu_r \Curl \bW +\mu_0 (\bM\cdot\Grad)\bH,\label{seq:dtuL}\\
	\Div \bU &=0\label{seq:divuL},\\
	\bW_t + (\bU\cdot \Grad)\bW - c_1\Delta \bW -c_2 \Grad\Div\bW +4 \nu_r \bW &= 2\nu_r \Curl \bU,\label{seq:dtwL}\\
	\bM &= \kappa_0\bH,\label{seq:ML}\\
	\Curl \bH &=0,\label{seq:curlhL}\\
	\Div (\bH + \bM)&= 0\label{seq:divbL},
	\end{align}
\end{subequations}
with boundary conditions 
\begin{equation*}
\bU=0,\quad \bW=0,\quad \bH\cdot\normal = \frac{1}{1+\kappa_0}\bh_a\cdot \bh,\quad \text{on } [0,T]\times\partial\dom.
\end{equation*}
The second objective of this article is to establish with rigorous arguments that under suitable assumptions on the initial data, we have that in the relaxation limit $ \tau\to 0$, 
$$\bm\rightarrow \kappa_0 \bH$$ and the sequence of solutions $(\bue,\bwe,\bme,\bhe)$ converges to a solution of~\eqref{eq:limitsys} when $\tau\to 0$. 
In fact, the following can be shown rigorously:
\begin{theorem}
	\label{thm:convequilibriumweaksol}
	Denote $\sole:=(\bue,\bwe,\bme,\bhe)$, $\bhe=\Grad\vfe$ a weak solution of system~\eqref{eq:rosenzweig} for a given $\tau>0$. Then as $\tau\to 0$, a subsequence of $\{\sole\}_{\tau>0}$ converges in $L^2([0,T]\times\dom)$ to a weak solution of~\eqref{eq:limitsys}.  
\end{theorem}

When the solution of the limiting system~\eqref{eq:limitsys} is smooth, which is the case for short times and smooth data and in two space dimensions for smooth enough data (this is shown in Appendix~\ref{S8}), we show a convergence rate in $\tau$ for the approximate solutions of~\eqref{eq:rosenzweig}:
\begin{theorem}
	\label{thm:convergencerate}
	If the solution of~\eqref{eq:limitsys} satisfies $\bU,\bW\in L^\infty(0,T;\mathrm{Lip}(\dom))$, $\bH=\Grad\Phi\in L^2(0,T;H^1(\dom))$, $\partial_t\bH\in L^2([0,T]\times\dom)$; and the initial data for~\eqref{eq:rosenzweig} and~\eqref{eq:limitsys} satisfy
	\begin{equation*}
	\norm{\bu_0-\bU_0}_{L^2(\dom)}^2 +\norm{\bw_0-\bW_0}_{L^2(\dom)}^2+\norm{\bm_0-\bM_0}_{L^2(\dom)}^2 \leq C\tau,
	\end{equation*}	
	then the solutions $\sole$ of~\eqref{eq:rosenzweig} converge as $\tau\to 0$ to the solution of the limiting system~\eqref{eq:limitsys} at the rate:
	\begin{multline*}
	\norm{\bue-\bU}_{L^2(\dom)}(t)+\norm{\bwe-\bW}_{L^2(\dom)}(t)+\norm{\bme-\bM}_{L^2(\dom)}(t)\\
	+\norm{\bhe-\bH}_{L^2(\dom)}(t)+\norm{\Grad(\bue-\bU)}_{L^2([0,t]\times\dom)}+\norm{\Grad(\bwe-\bW)}_{L^2([0,t]\times\dom)}\leq C\sqrt{\tau}(1+\exp(Ct)).
	\end{multline*}
	
\end{theorem}
The proof of the latter result uses the relative entropy method that was introduced by Dafermos~\cite{Dafermos1979,Dafermos1979b} and DiPerna~\cite{DiPerna1979} in the context of hyperbolic systems of conservation laws, see also~\cite{Dafermos1979a}. One constructs a suitable relative entropy functional that quantifies the difference between $\sole$ and the solution of the limiting system $\sol_0=(\bU,\bW,\bM,\bH)$ in $L^2$ and bounds its time evolution in terms of $\tau$. This is achieved by a careful estimation of all the resulting growth terms using the available bounds for the solution in an appropriate way. 

\subsection{Outline of this article.} The outline of this article is as follows: In Section \ref{S2} we present preliminaries and the energy law that governs the Rosensweig model. 
In Section \ref{S4} we present the proof of the global existence results for the Rosensweig model. Section \ref{S6} is devoted to the relaxation to equilibrium and the proof of Theorems~\ref{thm:convequilibriumweaksol} and~\ref{thm:convergencerate}. 
Finally, in Appendix \ref{S8} we discuss the global existence of classical solutions in $\R^2.$ 

\section{Preliminaries}\label{S2}
We start by introducing some notation: We denote by $\hdiv$ the space of divergence free $L^2(\dom)$ functions and by $\V$ the space of divergence free functions in $H^1_0(\dom)$ (these can be obtained as the closures of $C^\infty_0(\dom)\cap \{\Div u=0\}$ in $L^2(\dom)$ and $H^1_0(\dom)$ respectively (c.f.~\cite{Temam2001}):
\begin{equation*}
\begin{split}
\hdiv &= \{u\in L^2(\dom);\, \Div u = 0,\, u\cdot n|_{\partial \dom}=0\}, \\
\V &=\{u\in H^1(\dom);\, \Div u = 0,\, u|_{\partial \dom}=0\}. 
\end{split}
\end{equation*}
For a Banach space $X$ we let $C_w(0,T;X)$ be the space of functions that are weakly continuous in time, that is, if $v\in C_w(0,T;X)$, then for any $s\to t$,
\begin{equation*}
\langle v(t_n),g\rangle\to \langle v(t), g\rangle,\quad \text{for } s\to t, \quad \forall \, g\in X^*,
\end{equation*}
where we denoted by $X^*$ the dual space of $X$ and by $\langle\cdot,\cdot\rangle$ the dual product on $X,X^*$. Now we are ready to define a notion of weak solution for the Rosenzweig system \eqref{eq:rosenzweig}:
\begin{definition}[Distributional solution of the Rosensweig model]
	\label{def:weaksol}
	Let $T>0$, $\dom\subset\R^d$, $d=2,3$, a smooth, simply connected domain. We say that $\sol:=(\bu,\bw,\bm,\bh)$ is a global weak solution of the system \eqref{eq:rosenzweig} if the following conditions are satisfied:
	\begin{enumerate}[label=(\roman*)]
		\item The solution $\sol:=(\bu,\bw,\bm,\bh)$ satisfies the regularity requirements:
		\begin{align*}
		\bu&\in L^\infty(0,T;\hdiv)\cap L^2(0,T;\V)\cap C_w(0,T;L^2(\dom))\\
		\bw & \in L^\infty(0,T;L^2(\dom))\cap L^2(0,T;H^1_0(\dom))\cap C_w(0,T;L^2(\dom))\\
		\bm & \in L^\infty(0,T;L^2(\dom))\cap C_w(0,T;L^2(\dom))\\
		\bh & \in L^\infty(0,T;L^2(\dom))\cap C_w(0,T;L^2(\dom));
		\end{align*}
		\item the function $\bh$ is such that $\bh=\Grad \varphi$, where $\varphi\in L^\infty(0,T;H^1(\dom))$ and satisfies for all $\psi\in H^1(\dom)$,
		\begin{equation*}
		\int_{\dom} \Grad\varphi \cdot\Grad\psi dx =\int_{\dom} \left(\bh_a-\bm\right)\cdot \Grad \psi\, dx,
		\end{equation*}
		with Neumann boundary conditions: 
		\begin{equation*}
		\frac{\partial\varphi}{\partial n}=(\bh_a-\bm)\cdot\normal\quad \textrm{on }[0,T]\times\Gamma\quad\text{ with } \int_{\dom}\varphi(t,x) dx = 0;
		\end{equation*}
		\item Equations \eqref{seq:dtu}--\eqref{seq:dtm} hold weakly, that is for any test functions $\psi_i\in C^1_c([0,T)\times\dom)$ $i=1,\dots,4$ with vanishing trace on $\partial\dom$ and $\Div \psi_1(t,x) = 0$ for all $(t,x)\in[0,T]\times\dom$, 
		\begin{align*}
		&\lint \left[\bu\cdot\partial_t\psi_1+((\bu\cdot\Grad)\psi_1)\cdot\bu -(\nu+\nu_r)\Grad\bu:\Grad \psi_1 \right]dx dt + \int_{\dom} \bu_0\cdot\psi_1(0,x) dx\\ &\qquad=\lint\left[-2\nu_r \bw\cdot\Curl\psi_1 +\mu_0\left(\left(((\bm+\bh)\cdot\Grad)\psi_1\right)\cdot\bh \right)\right] dx dt\\
		&\int_{\dom}\bu(t,x)\cdot\Grad\psi_2(t,x) dx = 0,\quad \text{a.e. } t\in [0,T]\\
		&\lint\left[\bw\cdot\partial_t\psi_3+((\bu\cdot\Grad)\psi_3)\cdot\bw
		-c_1\Grad\bw:\Grad\psi_3 - c_2\Div\bw\, \Div\psi_3
		\right]dx dt + \int_{\dom} \bw_0\cdot\psi_3(0,x) dx \\
		&\qquad=\lint\left[  4\nu_r\bw\cdot\psi_3-2\nu_r \bu\cdot\Curl\psi_3 -\mu_0\left(\bm\times\bh\right)\cdot\psi_3\right] dx dt\\
		&\lint \left[\bm\cdot\partial_t\psi_4+(\bu\cdot\Grad)\psi_4\cdot\bm \right]dx dt + \int_{\dom} \bm_0\cdot\psi_4(0,x) dx\\ &\qquad=\lint\left[-(\bw\times\bm)\cdot\psi_4+\frac{1}{\tau}(\bm-\kappa_0\bh)\cdot\psi_4\right] dx dt
		\end{align*}
		where $\bu_0\in\hdiv$, $\bw_0,\bm_0\in L^2(\dom)$ are the initial conditions;
		
		\item The solution $\sol:=(\bu,\bw,\bm,\bh)$ satisfies the following energy inequality,
		\begin{equation}\label{eq:energy0}
		\int_{\dom} E(\sol)(t) dx +\int_0^t\int_{\dom} D(\sol)(s)\,dx ds\leq \int_{\dom} E(\sol)(0) dx +\mu_0 \int_0^t\int_{\dom}\partial_s\bh_a\cdot\bh \,dxds.
		\end{equation}
		where the energy $E$ is defined by
		\begin{equation}\label{eq:defE}
		E(\sol) = \frac{1}{2}\left(|\bu|^2 +|\bw|^2+ \frac{\mu_0}{\kappa_0}|\bm|^2+\mu_0|\bh|^2\right),
		\end{equation}
		and the dissipation functional $D$ is defined by
		\begin{equation}\label{eq:defD}
		D(\sol) = \left(\nu |\Grad \bu|^2 + c_1 |\Grad\bw|^2 + c_2|\Div \bw|^2+\nu_r|\Curl \bu -2\bw|^2+\frac{\mu_0}{\tau \kappa_0}|\bm-\kappa_0\bh|^2 \right).
		\end{equation}
		\end{enumerate}
\end{definition}
\begin{remark}\label{rem:energysmoothsol}
	For smooth solutions, the energy inequality~\eqref{eq:energy0} in fact becomes an equality and can be obtained by taking the inner product of equation~\eqref{seq:dtu} with $\bu$, equation~\eqref{seq:dtw} with $\bw$, equation~\eqref{seq:dtm} with $(\frac{\mu_0}{\kappa_0}\bm-\mu_0\bh)$ and the time-differentiated equation~\eqref{seq:divb} with $-\mu_0\varphi$, adding all of them, integrating over $\dom$, and then integrating by parts a couple of times.
\end{remark}

\subsection{Auxiliary results}
If $\bh$ does not have more spatial regularity than $L^2(\dom)$, the Kelvin force $\mu_0 (\bm\cdot\Grad)\bh$ in~\eqref{seq:dtu} is not well-defined, not even in a weak sense. Therefore, we used the following identity to define weak solutions in Definition~\ref{def:weaksol}, which allows us to make sense of the Kelvin force for $\bm$ and $\bh$ with little regularity.
\begin{lemma}\label{lem:integrableweaksol}
	Any sufficiently smooth solution of \eqref{eq:rosenzweig} satisfies
	\begin{equation}\label{eq:mhidentity}
	(\bm\cdot\Grad)\bh = \Div\left((\bm+\bh)\otimes \bh \right) -\frac{1}{2}\Grad \left(|\bh|^2\right).
	\end{equation}
\end{lemma}
\begin{proof}
	We write the differential operators in terms of their components:
	\begin{equation}\label{eq:calc1}
	\begin{split}
	\left((\bm\cdot\Grad)\bh\right)^{(j)} &= \sum_{i=1}^3 m^{(i)}\partial_i h^{(j)}\\
	&= \sum_{i=1}^3 \partial_i\left(m^{(i)} h^{(j)}\right) -\sum_{i=1}^3 \partial_i m^{(i)} h^{(j)} \\
	& = \Div\left(\bm h^{(j)}\right) -\Div \bm\, h^{(j)}.
	\end{split}
	\end{equation}
	We use equation \eqref{seq:divb} to replace the divergence of $\bm$ in the last expression:
	\begin{equation}\label{eq:calc2}
	\Div\left(\bm h^{(j)}\right) -\Div \bm\, h^{(j)} = \Div\left(\bm h^{(j)}\right) +\Div \bh\, h^{(j)} .
	\end{equation}
	By \eqref{seq:curlh}, $\bh$ is the gradient of a potential $\varphi$, therefore
	\begin{equation*}
	\begin{split}
	\Div \bh\, h^{(j)} &= \Delta \varphi \,\partial_j \varphi\\
	& = \sum_{i=1}^3\partial_{ii}^2 \varphi\, \partial_j \varphi\\
	& = \sum_{i=1}^3\left( \partial_i\left(\partial_i\varphi\,\partial_j \varphi\right) - \partial_i\varphi\,\partial_i\partial_j \varphi\right) \\
	& = \sum_{i=1}^3\left( \partial_i\left(\partial_i\varphi\,\partial_j \varphi\right) - \frac{1}{2}\partial_j\left|\partial_i \varphi\right|^2\right) \\
	& = \Div\left(\Grad\varphi\,\partial_j\varphi\right) - \frac{1}{2}\partial_j\left(\left|\Grad\varphi\right|^2\right)\\
	& = \Div\left(\bh\,h^{(j)}\right) - \frac{1}{2}\partial_j\left(|\bh|^2\right), 
	\end{split}
	\end{equation*}
	where we replaced $\Grad\varphi$ by $\bh$ again in the last equation. We combine the last calculation with \eqref{eq:calc1} and \eqref{eq:calc2} and obtain
	\begin{equation*}
	\begin{split}
	\left((\bm\cdot\Grad)\bh\right)^{(j)} &= \Div\left(\bm h^{(j)}\right) +\Div\left(\bh\,h^{(j)}\right) - \frac{1}{2}\partial_j\left(|\bh|^2\right)\\
	& = \Div\left((\bm+\bh) h^{(j)}\right)  - \frac{1}{2}\partial_j\left(|\bh|^2\right),
	\end{split}
	\end{equation*}
	which is the component form of \eqref{eq:mhidentity}.
\end{proof}	
\begin{remark}
	If we only require $\bm\in L^2(\dom)$, $\bh\in H^1(\dom)$, identity \eqref{eq:mhidentity} still holds in the sense of distributions for smooth enough and compactly supported in $\dom$ test functions $\psi$:
	\begin{equation}\label{eq:mhidentitydistributions}
	\int_{\dom}\left((\bm\cdot\Grad)\bh\right)\cdot\psi dx = -\int_{\dom} \left[\left(\left((\bm+\bh)\cdot \Grad\right)\psi\right)\cdot\bh -\frac{1}{2} |\bh|^2 \Div\psi\right] dx.
	\end{equation} 
	The right hand side of equation \eqref{eq:mhidentitydistributions} is bounded even if $\bm,\bh\in L^2(\dom)$ (and $\psi$ is smooth enough, i.e., in $C^1_c(\dom)$ with bounded derivatives). It therefore allows us to define weak solutions of the Rosensweig model \eqref{eq:rosenzweig} when $\bm$ and $\bh\in L^2(\dom)$ only.
\end{remark}
\section{Existence of global weak solutions}\label{S4}
In this section we prove the global existence of weak solutions to system~\eqref{eq:rosenzweig}, 
by constructing a sequence of solutions that satisfies an approximating system and by establishing  that the sequence converges to a solution of~\eqref{eq:rosenzweig}. As an approximating sequence, we use weak solutions of the regularized system,
\begin{subequations}\label{eq:approxsys}
	\begin{align}
	\bus_t +(\bus\cdot\Grad) \bus- (\nu+\nu_r)\Delta \bus+\Grad p^\sigma &= 2 \nu_r \Curl \bws +\mu_0 (\bms\cdot\Grad)\bhs,\label{seq:dtua}\\
	\Div \bus &=0\label{seq:divua},\\
	\bws_t + (\bus\cdot \Grad)\bws - c_1\Delta \bws -c_2 \Grad\Div\bws +4 \nu_r \bws &= 2\nu_r \Curl \bus+\mu_0 \bms\times \bhs,\label{seq:dtwa}\\
	\bms_t +(\bus\cdot \Grad)\bms-\sigma \Delta \bms &=\bws\times \bms-\frac{1}{\tau}(\bms-\kappa_0 \bhs),\label{seq:dtma}\\
	\Curl \bhs &=0,\label{seq:curlha}\\
	\Div (\bhs + \bms)&= 0\label{seq:divba},
	\end{align}
\end{subequations}
which is identical to~\eqref{eq:rosenzweig} up to the extra term $\sigma\Delta\bms$ in the magnetization equation~\eqref{seq:dtma}, where $\sigma>0$, that yields additional regularity for $\bms$ and so also for $\bhs$. In the literature, this regularized equation for $\bms$ is sometimes called a Bloch-Torrey type equation~\cite{Amirat2008,Torrey1956,Gaspari1966} because it was proposed by Torrey~\cite{Torrey1956} and can be seen as a generalization of the Bloch equations~\cite{Bloch1946} to describe situations in which the diffusion of the spin magnetic moment is not negligible. 

Since equation~\eqref{seq:dtma} is parabolic for $\sigma>0$, additional boundary conditions for $\bms$ have to be imposed in contrast to the case $\sigma=0$. A discussion of several reasonable choices of boundary conditions for $\bms$ can be found in~\cite[Section 2.3]{Nochetto2016}. We use the natural boundary conditions
\begin{equation}
\label{eq:boundaryconditionsm}
\Curl \bms\times\normal =0,\quad \Div\bms = 0,\quad \text{on } (0,T)\times\partial\dom,
\end{equation}
which allow obtaining an energy-stable system for the chosen boundary conditions for $\bhs$ (i.e. \eqref{eq:potentialbc}). The energy inequality is needed to show existence of weak solutions of the system~\eqref{eq:approxsys} using a Galerkin approximation, as it was done for a slightly different system by Amirat, Hamdache, Murat in~\cite{Amirat2008}. In particular, in their system, equation~\eqref{seq:divba} is replaced by
\begin{equation*}
\Div(\bhs+\bms)=f,\quad \bhs=\Grad\phis,
\end{equation*}
with boundary conditions 
\begin{equation*}
\frac{\partial\varphi}{\partial n}=0,\quad  \Curl \bms\times\normal =0,\quad \bms\cdot\normal = 0,\quad \text{on } (0,T)\times\partial\dom,
\end{equation*}
for $\phis$ and $\bms$ instead of~\eqref{eq:boundaryconditionsm}. Up to this difference, which implies that $\bms$ is sought in a different space, one defines weak solutions of~\eqref{eq:approxsys} in the same way as in~\cite{Amirat2008} (see upcoming Definition~\ref{def:approxweaksol}) and also existence of weak solutions is proved in the same way (the modifications needed due to the changed boundary conditions are sketched in Appendix~\ref{a:modifications}).
The solution $\bms(t)$ now lies for almost every $t\in [0,T]$ in the space
\begin{equation}
\label{eq:spaceform}
\mathcal{K}:=\{q\in L^2(\dom)\,|\,\Div q\in L^2(\dom),\,\Curl q\in L^2(\dom)\}=H(\Div)\cap H(\Curl),
\end{equation}
equipped with the inner product
\begin{equation*}
\langle q_1, q_2\rangle := \int_{\dom} q_1\cdot q_2\, dx +\int_{\dom} \Div q_1 \Div q_2\, dx + \int_{\dom}\Curl q_1\cdot \Curl q_2\, dx.
\end{equation*}

Then, weak solutions of~\eqref{eq:approxsys} are defined as follows:

\begin{definition}[Distributional solution of~\eqref{eq:approxsys},~\cite{Amirat2008}]
	\label{def:approxweaksol}	
	Let $T>0$, $\dom\subset\R^3$ a smooth, simply connected domain. We say that $\sol^\sigma:=(\bus,\bws,\bms,\bhs)$ is a global weak solution of the system \eqref{eq:approxsys} if the following conditions are satisfied:
	\begin{enumerate}[label=(\roman*)]
		\item The solution $\sol^\sigma=(\bus,\bws,\bms,\bhs)$ has the regularity properties:
		\begin{align*}
		\bus&\in L^\infty(0,T;\hdiv)\cap L^2(0,T;\V)\cap C_w([0,T];\hdiv)\\
		\bws & \in L^\infty(0,T;L^2(\dom))\cap L^2(0,T;H^1_0(\dom))\cap C_w([0,T];L^2(\dom))\\
		\bms & \in L^\infty(0,T;L^2(\dom))\cap L^2(0,T;\mathcal{K})\cap C_w([0,T];L^2(\dom))\\
		\bhs & \in L^\infty(0,T;L^2(\dom))\cap L^2(0,T;H^1(\dom));
		\end{align*}
		\item the function $\bhs$ is such that $\bhs=\Grad \phis$, where $\phis\in L^\infty(0,T;H^1(\dom))\cap L^2(0,T;H^2(\dom))$ solves the problem
		\begin{subequations}
			\label{eq:ellipticTorrey}
			\begin{align}
			-\Delta \phis &=\Div\bms ,\quad \text{in } [0,T]\times\dom,\label{seq:phiapprox}\\
			\frac{\partial\phis}{\partial \normal}&=(\bh_a-\bms)\cdot\normal,\quad\text{ on }[0,T]\times\partial\dom,\quad \int_{\dom}\phis\, dx =0,\quad \text{ in }(0,T),\label{eq:phiapproxbc}
			\end{align}
		\end{subequations}
		in the distributional sense;
		\item Equations \eqref{seq:dtua}--\eqref{seq:dtma} hold weakly, that is, for any test functions $\psi_1\in \V$, $\psi_2\in H^1_0(\dom)$, $\psi_3\in\mathcal{K}$, we have
		\begin{align}
		&\frac{d}{dt}\int_{\dom} \bus\cdot\psi_1 \, dx +\int_{\dom}[(\bus\cdot\Grad)\bus]\cdot\psi_1 \, dx +(\nu+\nu_r)\int_{\dom}\Grad\bus:\Grad\psi_1\, dx\notag\\
		&\qquad\qquad\qquad = \mu_0\int_{\dom}(\bms\cdot\Grad)\bhs\cdot\psi_1\, dx +2\nu_r\int_{\dom}\Curl\bws\cdot\psi_1\, dx,\quad\text{in } \mathcal{D}'((0,T)),\label{eq:weakapproxu}\\
		&\qquad \qquad \bus(0,\cdot)=\bu_0;\notag\\
		&\frac{d}{dt}\int_{\dom}\bws\cdot\psi_2 \,dx +\int_{\dom}[(\bus\cdot\Grad)\bws]\cdot\psi_2\, dx +c_1\int_{\dom}\Grad\bws:\Grad\psi_2\, dx + c_2\int_{\dom}\Div\bws\Div\psi_2\, dx\notag\\
		&\qquad\qquad\qquad=\mu_0\int_{\dom}(\bms\times\bhs)\cdot\psi_2\, dx
		+2\nu_r\int_{\dom}(\Curl\bus-2\bws)\cdot\psi_2\, dx,\quad\text{in } \mathcal{D}'((0,T)),\label{eq:weakapproxw}\\
		&\qquad \qquad \bws(0,\cdot)=\bw_0;\notag\\
		&\frac{d}{dt}\int_{\dom}\bms\cdot\psi_3 \,dx -\int_{\dom}[(\bus\cdot\Grad)\psi_3]\cdot\bms\, dx +
		\sigma\int_{\dom}\Curl\bms:\Curl\psi_3\, dx + \sigma\int_{\dom}\Div\bms\Div\psi_3\, dx\notag\\
		&\qquad\qquad\qquad= \int_{\dom}(\bws\times\bms)\cdot\psi_3\, dx-\frac{1}{\tau}\int_{\dom}(\bms-\kappa_0\bhs)\cdot\psi_3\, dx,\quad\text{in } \mathcal{D}'((0,T)),\label{eq:weakapproxm}\\
		&\qquad \qquad \bms(0,\cdot)=\bm_0;\notag
		\end{align}
		where $\bu_0\in\hdiv$, $\bw_0,\bm_0\in L^2(\dom)$ are the initial conditions.		
	\end{enumerate}
\end{definition}

\begin{remark}
	Using Lemma~\ref{lem:integrableweaksol}, the weak formulation for $\bus$ can be rewritten as
	\begin{multline*}
	\frac{d}{dt}\int_{\dom} \bus\cdot\psi_1 \, dx +\int_{\dom}[(\bus\cdot\Grad)\bus]\cdot\psi_1 \, dx +(\nu+\nu_r)\int_{\dom}\Grad\bus:\Grad\psi_1\, dx\\
	= -\mu_0\int_{\dom}\left(((\bms+\bhs)\cdot\Grad)\psi_1\right)\cdot\bhs\, dx  +2\nu_r\int_{\dom}\Curl\bws\cdot\psi_1\, dx,\quad\text{in } \mathcal{D}'((0,T)).
	\end{multline*}
\end{remark}
\subsection{Energy inequality}
In~\cite[Theorem 1]{Amirat2008}, it was proved that weak solutions as in Definition~\ref{def:approxweaksol} exist if $\bu_0\in\hdiv$, $\bw_0,\bm_0\in L^2(\dom)$ and that they satisfy in addition an energy inequality. The energy inequality proved there is slightly different from the one we are going to use in the following, but can be proved in the same way. Formally, one can use $\psi_1=\bus$, $\psi_2=\bws$ and $\psi_3=\frac{\mu_0}{\kappa_0}\bms-\mu_0\bhs$ as test functions in~\eqref{eq:weakapproxu}, \eqref{eq:weakapproxw}, and~\eqref{eq:weakapproxm} respectively, and add; take the derivative of~\eqref{seq:phiapprox}, multiply with $\mu_0\phis$, integrate over $\dom$ and add as well, to obtain, 
\begin{equation}\label{eq:energyapprox}
\begin{split}
&\frac{1}{2}\frac{d}{dt}\int_{\dom}\left(|\bus|^2+|\bws|^2+\frac{\mu_0}{\kappa_0}|\bms|^2+\mu_0|\bhs|^2\right)dx\\
&\qquad+\int_{\dom}\left(\nu |\Grad \bus|^2 + c_1 |\Grad\bws|^2 + c_2|\Div \bws|^2+\nu_r|\Curl \bus -2\bws|^2+\frac{\mu_0}{\tau \kappa_0}|\bms-\kappa_0\bhs|^2\right) dx \\
&\qquad +\sigma\mu_0\int_{\dom}\left(\frac{1}{\kappa_0}|\Curl\bms|^2+\left(\frac{1}{\kappa_0}+1\right)|\Div\bms|^2\right) dx  = \mu_0\int_{\dom}\partial_t\bh_a\cdot\bhs dx,
\end{split}
\end{equation}
where we have also used that
\begin{equation*}
\Delta\bms =-\Curl^2\bms+\Grad\Div\bms,
\end{equation*}
and that this identity combined with~\eqref{seq:phiapprox}, and the boundary conditions for $\bms$ and $\bhs$, yields
\begin{align*}
\int_{\dom}\Delta\bms\cdot\bhs \, dx & = -\int_{\dom}(\Curl^2\bms-\Grad\Div\bms)\cdot\bhs\, dx\\
& = -\int_{\dom}\Curl\bms\cdot\Curl\bhs\, dx -\int_{\dom}\Div\bms\Div\bhs\, dx\\
&\qquad+\int_{\partial\dom}(\Curl\bms\times\normal)\cdot\bhs\,dS+\int_{\partial\dom}\Div\bms (\bhs\cdot\normal)\, dS\\
& = -\int_{\dom}\Div\bms\Div\bhs\, dx\\
& = \int_{\dom}|\Div\bms|^2\, dx.
\end{align*}
Since all of this has to be done at the level of the Galerkin approximation first, then passing to the limit in the approximation,~\eqref{eq:energyapprox} holds only as an inequality.

\begin{remark}
	The existence proof in~\cite{Amirat2008} also shows that $\bms$ satisfies
	\begin{multline}
	\label{eq:mineq}
	\frac{1}{2}\int_{\dom}|\bms|^2(t) dx +\sigma\int_0^t\!\!\int_{\dom}\left(|\Curl\bms|^2+|\Div\bms|^2\right)dxds \\\leq \frac{1}{2}\int_{\dom}|\bms_0|^2 dx -\frac{1}{\tau}\int_0^t\!\!\int_{\dom}\bms\cdot(\bms-\kappa_0\bhs)dx ds
	\end{multline}
\end{remark}

\subsection{Renormalized solutions}
In the proof of Proposition~\ref{prop:strongconvm}, we will make use of the fact that the magnetization equation~\eqref{seq:dtm} has the structure of a transport equation. In fact, we prove here that $\bm$ is `renormalized' as in the sense of DiPerna and Lions~\cite{DiPernaLions1989,Lions1998book,Novotny2004,Feireisl2004}.
We will need the following lemma:

\begin{lemma}
	\label{lem:transport}
	Let $\bm\in L^\infty(0,T;L^2(\dom))$ be a distributional solution of
	\begin{equation}\label{eq:mtransport}
	\partial_t \bm +\left(\bu\cdot\Grad\right)\bm = \bw\times\bm -\frac{1}{\tau}(\bm-\kappa_0\bh),
	\end{equation} 
	for given $\bu\in L^\infty(0,T;\hdiv)\cap L^2(0,T;\V)$, $\bw\in L^\infty(0,T;L^2(\dom))\cap L^2(0,T;H^1_0(\dom))$ and $\bh\in L^\infty(0,T;L^2(\dom))$. Then the components $m^{(i)}$ of $\bm$ satisfy for any $b\in C^1(\R)$ with $b'(\cdot)$ bounded, and $b(0)=0$,
	\begin{equation*}
	\partial_t b(m^{(i)}) + \bu\cdot\Grad b(\mi) = b'(\mi)(\bw\times\bm)^{(i)}-\frac{1}{\tau} b'(\mi)(\mi-\kappa_0 \hi),
	\end{equation*}
	in the sense of distributions.
\end{lemma}
\begin{proof}
	First we see that we can extend $\bu$, $\bw$ and $\bm$ by zero outside $\dom$ to make them functions of the whole $\R^d$, since $\bw$ and $\bu$ have zero trace and $\bm$ is assumed to be a function of $L^2(\dom)$ only with no requirements on the boundary traces. Since \eqref{eq:mtransport} holds in the sense of distributions, we can use a mollifier $\omega_\epsilon(x)=\epsilon^{-d}\omega(x/\epsilon)$, where $\omega\in C^\infty_c(\R^d)$ is supported on $[-1,1]^d$ and normalized $\int_{\R^d} \omega dx = 1$, as a test function in one component $\mi$ of $\bm$. Denoting $\mi_{\eps}:=\mi\ast\omega_\eps$, $\hi_{\eps}:=\hi\ast\omega_\eps$, we observe that the components $\mi_\eps$ of $\bm_\eps$ satisfy the following equation pointwise in $(t,x)$:
	\begin{equation*}
	\partial_t \mi_\eps + (\bu\cdot\Grad \mi)\ast\omega_\eps = (\bw\times\bm)^{(i)}\ast\omega_\eps - \frac{1}{\tau}(\mi_\eps-\kappa_0 \hi_\eps).
	\end{equation*}	
	We can thus multiply this equation by $b'(\mi_\eps)$ and apply the chain rule:
	\begin{equation*}
	\partial_t b(\mi_\eps) + b'(\mi_\eps)(\bu\cdot\Grad \mi)\ast\omega_\eps = b'(\mi_\eps) (\bw\times\bm)^{(i)}\ast\omega_\eps - \frac{1}{\tau}b'(\mi_\eps)(\mi_\eps-\kappa_0 \hi_\eps).
	\end{equation*}	
	The term $b'(\mi_\eps)(\bu\cdot\Grad \mi)\ast\omega_\eps$ can be written as
	\begin{equation*}
	\begin{split}
	b'(\mi_\eps)(\bu\cdot\Grad \mi)\ast\omega_\eps &= \bu\cdot\Grad b(\mi_\eps)+b'(\mi_\eps)\left((\bu\cdot\Grad \mi)\ast\omega_\eps - \bu\cdot\Grad\mi_\eps\right)\\
	& :=  \bu\cdot\Grad b(\mi_\eps)+b'(\mi_\eps) R_\eps,
	\end{split}
	\end{equation*}	
	hence for any $\psi\in C^\infty_c(\R^d)$
	\begin{multline*}
	\int_{\R^d}b(\mi_\eps(t,x))\psi(x) dx - \int_{\R^d} b(\mi_\eps(0,x))\psi(x) dx  -\int_0^t\int_{\R^d} b(\mi_\eps)\bu\cdot \Grad\psi(x)  dx ds \\
	= \int_0^t\int_{\R^d}\left( b'(\mi_\eps) (\bw\times\bm)^{(i)}\ast\omega_\eps - \frac{1}{\tau}b'(\mi_\eps)\left(\mi_\eps-\kappa_0 \hi_\eps\right)\right)\psi dx ds\\
	- \int_0^t\int_{\R^d}b'(\mi_\eps) R_\eps \psi dx ds.
	\end{multline*}
	Since $b\in C^1(\R)$ with bounded derivative, and $\mi,\hi\in L^\infty(0,T;L^2(\R^d))$, and $\bu,\bw\in L^2(0,T;H^1(\dom))$, we can pass to the limit in all the terms except the term involving $R_\eps$. That one converges to zero in $L^1_{\text{loc}}([0,\infty)\times\R^d)$ by~\cite[Lemma 2.3]{Lions1996}, therefore, since $b'(\cdot)$ is bounded and $\psi$ is compactly supported, the term involving $R_\eps$ converges to zero. We obtain in the limit $\eps\rightarrow 0$,
	\begin{multline}\label{eq:neededlater}
	\int_{\R^d}b(\mi(t,x))\psi(x) dx - \int_{\R^d} b(\mi(0,x))\psi(x) dx  -\int_0^t\int_{\R^d} b(\mi)\bu\cdot \Grad\psi(x)  dx ds \\
	= \int_0^t\int_{\R^d}\left( b'(\mi) (\bw\times\bm)^{(i)} - \frac{1}{\tau}b'(\mi)\left(\mi-\kappa_0 \hi\right)\right)\psi dx ds.
	\end{multline}
	Instead of a test function $\psi$ only depending on space, we could also have used a test function $\psi\in C^\infty_c([0,\infty)\times\R^d)$ and integrated by parts in the term involving the time derivative of $b(\mi_\eps)$. In that case, we obtain, after passing to the limit $\eps\rightarrow 0$,
	\begin{multline*}
	\int_0^\infty\int_{\R^d}b(\mi)\partial_t\psi dxdt +\int_{\R^d} b(\mi(0,x))\psi(x) dx  +\int_0^\infty\int_{\R^d} b(\mi)\bu\cdot \Grad\psi  dx dt \\
	= -\int_0^\infty\int_{\R^d}\left( b'(\mi) (\bw\times\bm)^{(i)} - \frac{1}{\tau}b'(\mi)\left(\mi-\kappa_0 \hi\right)\right)\psi dx dt,
	\end{multline*}
	which is what we wanted to prove (since $\mi$, $\bu$ and $\bw$ are all zero outside $\dom$, we can replace the integration over $\R^d$ by the integration over $\dom$).
\end{proof}	
Using that $\bm\in L^\infty(0,T;L^2(\dom))$, we can use the previous lemma to prove the following special case:
\begin{lemma}
	\label{lem:transport2}
	Let $\bm\in L^\infty(0,T;L^2(\dom))$ be a distributional solution of
	\begin{equation*}
	\partial_t \bm +\left(\bu\cdot\Grad\right)\bm = \bw\times\bm -\frac{1}{\tau}(\bm-\kappa_0\bh),
	\end{equation*} 
	for given $\bu\in L^\infty(0,T;\hdiv)\cap L^2(0,T;\V)$, $\bw\in L^\infty(0,T;L^2(\dom))\cap L^2(0,T;H^1_0(\dom))$, $\bh\in L^\infty(0,T;L^2(\dom))$ and initial data $\bm_0\in L^2(\dom)$. Then $\bm$ satisfies for almost any $t\in [0,T]$,
	\begin{equation}
	\label{eq:energyidm}
	\frac{1}{2}\int_{\dom}|\bm(t,x)|^2 dx = \frac{1}{2}\int_{\dom}|\bm_0|^2 dx -\frac{1}{\tau}\int_0^t\int_{\dom}(|\bm|^2-\kappa_0 \bh\cdot\bm) dx ds.
	\end{equation}
\end{lemma}
\begin{proof}
	The starting point is identity \eqref{eq:neededlater}. We would like to use $b(y)=y^2/2$ as a test function there and sum the equations for the components $\mi$ over $i=1,2 (3)$. However, such a $b$ does not satisfy the assumptions of Lemma \ref{lem:transport} because it does not have bounded derivative which could potentially make the term involving $(\bw\times\bm)^{(i)}$ unbounded. Nonetheless, because the cross product is orthogonal to the two vectors involving it, that term would formally be zero when using $b'(\mi)=\mi$ after summing over $\mi$. This indicates that to prove rigorously that we can use $b(\mathbf{y})=\sum_{i=1}^d b(y^{(i)})=\sum_{i=1}^d  (y^{(i)})^2/2$ as a test function, we should use approximations of it with gradient parallel to $\bm$. Given a finite $0<K<\infty$, let $\bik$ be the function with $\bik(0)=0$ and derivative
	\begin{equation*}
	(\bik)'(\mi)=\mi\cdot \mathbbm{1}_{|\bm|\leq K} +K\frac{\mi}{|\bm|} \mathbbm{1}_{|\bm|>K}=\begin{cases}
	\mi,&\quad |\bm|\leq K,\\
	K\cdot\frac{\mi}{|\bm|},&\quad |\bm|>K, 
	\end{cases}
	\end{equation*}
	where $|\bm|=\sqrt{\sum_{i=1}^d (\mi)^2}$. The derivative of $\bik$ is obviously bounded by definition, however, it might be discontinuous. We can therefore convolve it with some smooth mollifier $\omega_\delta(x)=\delta^{-1}\omega(x/\delta)$ (as defined in Lemma \ref{lem:transport} but in $\R$ instead of $\R^d$) to make it smooth and see that \eqref{eq:neededlater} holds for the primitive of the regularized $(\bikd)':=(\bik)'\ast\omega_{\delta}$,
	\begin{multline*}
	\int_{\R^d}\bikd(\mi(t,x))\psi(x) dx - \int_{\R^d} \bikd(\mi(0,x))\psi(x) dx  -\int_0^t\int_{\R^d} \bikd(\mi)\bu\cdot \Grad\psi(x)  dx ds \\
	= \int_0^t\int_{\R^d}\left( (\bikd)'(\mi) (\bw\times\bm)^{(i)} - \frac{1}{\tau}(\bikd)'(\mi)\left(\mi-\kappa_0 \hi\right)\right)\psi dx ds.
	\end{multline*}
	Since $|\bikd(\mi)|\leq C (|\mi| +|\mi|^2\mathbbm{1}_{|\mi|\leq K})$ and $|(\bikd)'(\mi)|\leq C K$ uniformly in $\delta>0$, $\mi,\hi\in L^\infty(0,T;L^2(\dom))$, $\bu\in L^\infty(0,T;L^2(\dom))$ and $\bw\times\bm\in L^\gamma([0,T]\times\dom)$ for some $\gamma>1$ (using Sobolev embeddings for $\bw$), all the quantities in the last expression are uniformly integrable in $\delta>0$ and we can pass to $\delta\rightarrow 0$ to obtain, by the properties of convolution
	\begin{multline*}
	\int_{\R^d}\bik(\mi(t,x))\psi(x) dx - \int_{\R^d} \bik(\mi(0,x))\psi(x) dx  -\int_0^t\int_{\R^d} \bik(\mi)\bu\cdot \Grad\psi(x)  dx ds \\
	= \int_0^t\int_{\R^d}\left( (\bik)'(\mi) (\bw\times\bm)^{(i)} - \frac{1}{\tau}(\bik)'(\mi)\left(\mi-\kappa_0 \hi\right)\right)\psi dx ds.
	\end{multline*}
	Now, let us choose a smooth test function $\psi\in C^\infty_c(\R^d)$ that satisfies $\psi(x)\equiv 1$ inside $\dom$, which implies that the convective term vanishes and we are left with
	\begin{multline*}
	\int_{\dom}\bik(\mi(t,x)) dx - \int_{\dom} \bik(\mi(0,x)) dx  \\
	= \int_0^t\int_{\dom}\left( (\bik)'(\mi) (\bw\times\bm)^{(i)} - \frac{1}{\tau}(\bik)'(\mi)\left(\mi-\kappa_0 \hi\right)\right) dx ds.
	\end{multline*}
	Now sum these identities over $i=1,\dots, d$, define $\mathbf{b}_K(\bm)=\sum_{i=1}^d \bik(\mi)$ and notice that $((b_K^{(1)})', \dots, (b_K^{(d)})')^\top$ is parallel to $\bm$ or zero, which means that 
	\begin{equation*}
	\sum_{i=1}^d (\bik)'(\mi)(\bw\times\bm)^{(i)} = 0,\quad \text{a.e. } (t,x)\in [0,\infty)\times\R^d,
	\end{equation*} 
	and thanks to the integrability properties on $\bw$ and $\bm$ also in $L^1([0,\infty)\times\R^d)$. So,
	\begin{equation*}
	\int_{\dom}\mathbf{b}_K(\bm(t,x)) dx - \int_{\dom} \mathbf{b}_K(\bm_0) dx = -\frac{1}{\tau}\sum_{i=1}^d\int_0^t\int_{\dom} (\bik)'(\mi)\left(\mi-\kappa_0 \hi\right) dx ds.
	\end{equation*}
	All the quantities in the last identity are uniformly integrable with respect to $K>0$, hence we can pass $K\rightarrow\infty$, and get, since $\mathbf{b}_K(\bm)\rightarrow|\bm|^2/2$ and $((b_K^{(1)})', \dots, (b_K^{(d)})')^\top\rightarrow\bm$, as $K\rightarrow\infty$, equation \eqref{eq:energyidm}.
\end{proof}
\subsection{Passing to the limit in the approximating sequence}
We are now ready to prove existence of weak solutions of \eqref{seq:dtu}-\eqref{seq:divb}: 
\begin{proposition}
	\label{prop:strongconvm}
	Let $\{\sigma_n\}_{n\in\N}$ be a sequence of nonnegative parameters such that $\sigma_n\stackrel{n\to\infty}{\longrightarrow}0$ and denote $\sol_n=(\bu_n,\bw_n,\bm_n, \bh_n,\varphi_n):=\sol^{\sigma_n}=(\busn,\bwsn,\bmsn,\bhsn,\phisn)$ the solutions of~\eqref{eq:approxsys} with $\sigma=\sigma_n$. 
	Assume that the initial data $\bu_0,\bw_0$ and $\bm_{0}$ satisfy~\eqref{eq:init}, $\bh_a\in L^\infty(0,T;L^2(\dom))\cap H^1(0,T;L^2(\dom))$ with $\Div \bh_a=0$ a.e. and that the boundary conditions~\eqref{eq:bc1},~\eqref{eq:potentialbc} and~\eqref{eq:boundaryconditionsm} are satisfied.
	Then, as $n\to\infty$, a subsequence of $\{(\bu_n,\bw_n,\bm_n, \bh_n,\varphi_n)\}_{n\in\N}$ converges to a weak solution of \eqref{eq:rosenzweig} as in Definition \ref{def:weaksol}.
\end{proposition}
\begin{proof}
	From the energy estimate~\eqref{eq:energyapprox}, we obtain that $\sol_n$ satisfies, uniformly in $n$,
	\begin{align}
	\label{eq:uniformbounds}
	\begin{split}
	&\{\bu_n\}_{n\in\N}\subset L^\infty(0,T; L^2_{\text{div}}(\dom))\cap L^2(0,T;H^1_{\text{div}}(\dom)),\\
	&\{\bw_n\}_{n\in\N} \subset L^\infty(0,T; L^2(\dom))\cap L^2(0,T;H^1(\dom)),\\
	&\{\bm_n\}_{n\in\N}\subset  L^\infty(0,T; L^2(\dom)),\\
	&\{\bh_n\}_{n\in\N} \subset L^\infty(0,T; L^2(\dom)),\\
	&\{\varphi_n\}_{n\in\N}\subset L^\infty(0,T; H^1(\dom)),\\
	&\{\sqrt{\sigma_n}\Curl\bm_n\}_{n\in\N},\, \{\sqrt{\sigma_n}\Div\bm_n\}_{n\in\N}\subset L^2([0,T]\times\dom),
	\end{split}
	\end{align}	
	
	Hence, from the Banach-Alaoglu theorem, we obtain that $(\bu_n,\bw_n,\bm_n, \bh_n,\varphi_n)_{n\in\N}$ has a weak$^*$ convergent subsequence, which we denote by $n$ again for convenience, such that as $n\rightarrow\infty$,
	\begin{align}\label{eq:weakconv}
	\begin{split}
	\bu_n\weakstar \bu,&\quad \text{in } L^\infty(0,T; L^2(\dom))\cap L^2(0,T;H^1(\dom)),\\
	\bw_n\weakstar \bw,&\quad \text{in } L^\infty(0,T; L^2(\dom))\cap L^2(0,T;H^1(\dom)),\\
	\bm_n\weakstar \bm,&\quad \text{in } L^\infty(0,T; L^2(\dom)),\\
	\bh_n\weakstar \bh,&\quad \text{in } L^\infty(0,T; L^2(\dom)),\\
	\varphi_n\weakstar \varphi,&\quad \text{in } L^\infty(0,T; H^1(\dom)),	
	\end{split}
	\end{align}	
	where the limits satisfy
	\begin{equation*}
	\begin{split}
	&\norm{\bu}_{L^\infty([0,T];L^2(\dom))\cap L^2([0,T];H^1(\dom))}\leq C,\quad \norm{\bw}_{L^\infty([0,T];L^2(\dom))\cap L^2([0,T];H^1(\dom))}\leq C,\\ &\norm{\bm}_{L^\infty([0,T];L^2(\dom))}\leq C, \quad \norm{\varphi}_{L^\infty([0,T];H^1(\dom))}\leq C, \quad
	\norm{\bh}_{L^\infty([0,T];L^2(\dom))}\leq C.
	\end{split}
	\end{equation*}
		Combining the equations for $\bw_n$ and $\bm_n$, \eqref{seq:dtwa} and \eqref{seq:dtma} respectively, with the uniform bounds~\eqref{eq:uniformbounds}, we obtain by Sobolev embeddings that
	\begin{equation}\label{eq:timederconv}
	%\{\partial_t\bu_n\}_{n\in \N},
	\{\partial_t\bw_n\}_{n\in\N}, \{\partial_t\bm_n\}_{n\in\N}\subset L^p(0,T;W^{-k,q}(\dom)),
	\end{equation}
	for $p,q>1$ and some $k>0$ large enough. Combining the equation~\eqref{seq:dtua} with the uniform bounds~\eqref{eq:uniformbounds}, we obtain for the velocity field $\bu_n$ 
	\begin{equation}
	\label{eq:timederconv2}
	\{\partial_t\bu_n\}_{n\in \N} %,
	%\{\partial_t\bw_n\}_{n\in\N}, \{\partial_t\bm_n\}_{n\in\N}
	\subset L^2(0,T;H^{-s}_{\text{div}}(\dom)),\quad \text{for}\, s\geq 1\, \text{ large enough,}
	\end{equation}
	where $H^{-s}_{\text{div}}(\dom)$ denotes the dual space of $H^s_{\text{div}}(\dom):=H^1_{\text{div}}(\dom)\cap H^s(\dom)$, $s\geq 1$, the space of divergence free functions in $H^s(\dom)$ with vanishing trace.
	% (see for example p.\ 51-52 and p.\ 90-92 in~\cite{foias_manley_rosa_temam_2001} for definitions and properties of these spaces), 
	Since $H^1(\dom)\stackrel{c}{\hookrightarrow} L^2(\dom)\hookrightarrow W^{-k,p}(\dom)$ and $H^1_{\text{div}}(\dom)\stackrel{c}{\hookrightarrow}L^2_{\text{div}}(\dom)\hookrightarrow H^{-s}_{\text{div}}(\dom)$, where the first embeddings are compact and the second ones continuous, we can use Aubin-Lions' Lemma~\cite{Simon1986} to obtain
	\begin{equation}\label{eq:strongconvuw}
	\bu_n\rightarrow \bu,\quad \bw_n\rightarrow\bw,\quad \text{ in } L^2([0,T]\times\dom),
	\end{equation}
	as $n\rightarrow \infty$ up to a subsequence.
	We use these convergence properties~\eqref{eq:weakconv} and~\eqref{eq:strongconvuw}
	% and Assumption \ref{item:initdata} (suitable approximation of the initial data) 
	to pass to the limit in the weak formulations of \eqref{seq:dtua}, \eqref{seq:dtwa} and \eqref{seq:dtma}:
	\begin{align*}
	0& =\lint \left[\bu_n\cdot\partial_{t}\psi_1+((\bu_n\cdot\Grad)\psi_1)\cdot\bu_n -(\nu+\nu_r)\Grad\bu_n:\Grad \psi_1 \right]dx dt + \int_{\dom} \bu_{0}\cdot\psi_1(0,x) dx\\ &\qquad+\lint\left[2\nu_r \bw_n\cdot\Curl\psi_1 -\mu_0\left(((\bm_n+\bh_n)\cdot\Grad)\psi_1\right)\cdot\bh_n \right] dx dt\\
	&\quad\xrightarrow{n\rightarrow\infty}  \quad\lint \left[\bu\cdot\partial_t\psi_1+((\bu\cdot\Grad)\psi_1)\cdot\bu -(\nu+\nu_r)\Grad\bu:\Grad \psi_1 \right]dx dt + \int_{\dom} \bu_0\cdot\psi_1(0,x) dx\\ 
	&\hphantom{\quad\xrightarrow{n\rightarrow\infty}  \quad}\qquad+\lint\left[2\nu_r \bw\cdot\Curl\psi_1 -\mu_0\overline{\left(((\bm+\bh)\cdot\Grad)\psi_1\right)\cdot\bh}\right] dx dt
	\end{align*}
	where we denoted by $\overline{\left(((\bm+\bh)\cdot\Grad)\psi_1\right)\cdot\bh}$ the weak limit of $\left(((\bm_n+\bh_n)\cdot\Grad)\psi_1\right)\cdot\bh_n$,
	\begin{equation*}
	0=\int_{\dom}\bu_n(t,x)\cdot\Grad\psi_2(t,x) dx \quad\xrightarrow{n\rightarrow\infty}\quad \int_{\dom}\bu(t,x)\cdot\Grad\psi_2(t,x) dx,
	\end{equation*}
	\begin{align*}
	0&=\lint\left[\bw_n\cdot\partial_{t}\psi_3+((\bu_n\cdot\Grad)\psi_3)\cdot\bw_n
	-c_1\Grad\bw_n:\Grad\psi_3 - c_2\Div\bw_n\, \Div\psi_3
	\right]dx dt \\
	&\qquad +\lint\left[  -4\nu_r\bw_n\cdot\psi_3+2\nu_r \bu_n\cdot\Curl\psi_3 +\mu_0\left(\bm_n\times\bh_n\right)\cdot\psi_3\right] dx dt + \int_{\dom} \bw_{0}\cdot\psi_3(0,x) dx\\
	&\quad\xrightarrow{n\rightarrow\infty}\quad \lint\left[\bw\cdot\partial_t\psi_3+((\bu\cdot\Grad)\psi_3)\cdot\bw
	-c_1\Grad\bw:\Grad\psi_3 - c_2\Div\bw\, \Div\psi_3
	\right]dx dt  \\
	&\hphantom{\quad\xrightarrow{n\rightarrow\infty}}\qquad +\lint\left[  -4\nu_r\bw\cdot\psi_3+2\nu_r \bu\cdot\Curl\psi_3 +\mu_0\left(\overline{\bm\times\bh}\right)\cdot\psi_3\right] dx dt+ \int_{\dom} \bw_0\cdot\psi_3(0,x) dx.
	\end{align*}
	Note that  we denoted the weak limit of $\bm_n\times\bh_n$ by $\overline{\bm\times\bh}$, and 
	\begin{align*}
	0&=\lint \left[\bm_n\cdot\partial_{t}\psi_4+(\bu_n\cdot\Grad)\psi_4\cdot\bm_n
	-\sigma_n\Curl\bm_n\cdot\Curl\psi_4-\sigma_n\Div\bm_n\Div\psi_4\right]dx dt \\ &\qquad+ \int_{\dom} \bm_{0}\cdot\psi_4(0,x) dx+\lint\left[(\bw_n\times\bm_n)\cdot\psi_4-\frac{1}{\tau}(\bm_n-\kappa_0\bh_n)\cdot\psi_4\right] dx dt\\
	&\quad\xrightarrow{n\rightarrow\infty}\quad \lint \left[\bm\cdot\partial_t\psi_4+(\bu\cdot\Grad)\psi_4\cdot\bm  \right]dx dt + \int_{\dom} \bm_0\cdot\psi_4(0,x) dx\\ &\hphantom{\quad\xrightarrow{n\rightarrow\infty}}\qquad+\lint\left[(\bw\times\bm)\cdot\psi_4-\frac{1}{\tau}(\bm-\kappa_0\bh)\cdot\psi_4\right] dx dt.
	\end{align*}%-\sigma_n\Grad\bm:\Grad \psi_4
	Passing to the limit in the equation for $\varphi_n$,~\eqref{seq:phiapprox}, and the equation for $\bh_n$, $\bh_n=\Grad\varphi_n$, we obtain that $\bh=\Grad\varphi$ in $L^2(\dom)$ and that $\varphi$ satisfies the equation
	\begin{equation}\label{eq:weakphi}
	\int_{\dom}\Grad\varphi\cdot\Grad\psi dx = \int_{\dom}\bh\cdot\Grad\psi dx = \int_{\dom}(\bh_a-\bm)\cdot\Grad\psi dx 
	\end{equation}
	for any test function $\psi\in H^1(\dom)$. In essence, to arrive at the conclusion that $(\bu,\bw,\bm,\bh,\varphi)$ is a weak solution of \eqref{eq:rosenzweig}, we need to show that the weak limits
	\begin{equation}\label{eq:weakstrong}
	\overline{\left(((\bm+\bh)\cdot\Grad)\psi_1\right)\cdot\bh}= {\left(((\bm+\bh)\cdot\Grad)\psi_1\right)\cdot\bh},\quad\text{and}\quad \overline{\bm\times\bh} = {\bm\times\bh},
	\end{equation}
	on $[0,T]\times\dom$ in the sense of distribution.
	This can be achieved by showing that either a subsequence of $\{\bm_n\}_{n\in\N}$ or $\{\bh_n\}_{n\in\N}$ converges strongly in $L^2([0,T]\times\dom)$. We will show in the following that the sequence $\{\bm_n\}_{n\in\N}$ converges strongly (which will imply that also $\{\bh_n\}_{n\in\N}$ converges strongly in $L^2$).
	
	First we notice that we can use $\varphi$ as a test function in \eqref{eq:weakphi}, since it is in $L^\infty(0,T;H^1(\dom))$. This yields
	\begin{equation}	\label{eq:phiid1}
	\int_{\dom}|\Grad\varphi|^2 dx = \int_{\dom}|\bh|^2 dx = \int_{\dom}(\bh_a-\bm)\cdot\Grad\varphi dx= \int_{\dom}(\bh_a-\bm)\cdot\bh\, dx,\quad \text{a.e. } t. 
	\end{equation}
	where we have substituted $\bh$ for $\Grad\varphi$. We can also use $\varphi_n$ as a test function at the level of approximation \eqref{seq:phiapprox}:
	\begin{equation*}
	\int_{\dom}|\Grad\varphi_n|^2 dx = \int_{\dom}|\bh_n|^2 dx = \int_{\dom}(\bh_a-\bm_n)\cdot\Grad\varphi_n dx= \int_{\dom}(\bh_a-\bm_n)\cdot\bh_n  dx. 
	\end{equation*}
	Passing to the limit $n\rightarrow\infty$ in this last expression, we get
	\begin{equation}\label{eq:phiid2}
	\int_{\dom}\overline{|\Grad\varphi|^2} dx = \int_{\dom}\overline{|\bh|^2} dx = \int_{\dom}\left(\bh_a\cdot \Grad\varphi-\overline{\bm\cdot\Grad\varphi}\right) dx= \int_{\dom}\left(\bh_a\cdot \bh-\overline{\bm\cdot\bh}\right) dx, \quad \text{a.e. } t
	\end{equation}
	where $\overline{|\Grad\varphi|^2}$ and $\overline{|\bh|^2}$ denote the weak limits of the sequences $|\Grad\varphi_n|^2$ and $|\bh_n|^2$ respectively.
	Combining \eqref{eq:phiid1} and \eqref{eq:phiid2}, and rearranging (the terms containing $\int_{\dom}\bh_a\cdot\bh dx$ are identical), we get the expression
	\begin{equation}
	\label{eq:mhprod}
	\int_{\dom}\left(|\bh|^2  +\bm\cdot\bh\right) dx = \int_{\dom}\left(\overline{|\bh|^2}+\overline{\bm\cdot\bh}\right) dx.
	\end{equation}
	This looks not very exciting at first sight, but we will see later how it is useful.
Next, we note that from~\eqref{eq:mineq}, we have,
	\begin{equation*}
	\frac{1}{2}\int_{\dom}|\bm_n(t,x)|^2 dx \leq \frac{1}{2}\int_{\dom}|\bm_{0}|^2 dx -\frac{1}{\tau}\int_0^{t}\int_{\dom}\left(|\bm_n|^2 - \kappa_0\bh_n\cdot \bm_n\right) dx ds.
	\end{equation*}
	As $n\rightarrow\infty$, this converges to
	\begin{equation}\label{eq:limitenergyineq}
	\frac{1}{2}\int_{\dom}\overline{|\bm(t,x)|^2} dx \leq \frac{1}{2}\int_{\dom}|\bm_{0}|^2 dx -\frac{1}{\tau}\int_0^t\int_{\dom}\left(\overline{|\bm|^2} - \kappa_0\overline{\bh\cdot \bm}\right) dx ds,
	\end{equation}
	where we denoted the weak limit of $|\bm_n|^2$ by $\overline{|\bm|^2}$.
	By Lemma \ref{lem:transport2}, the weak limit $\bm$ satisfies
	\begin{equation*}
	\frac{1}{2}\int_{\dom}|\bm(t,x)|^2 dx = \frac{1}{2}\int_{\dom}|\bm_0|^2 dx -\frac{1}{\tau}\int_0^t\int_{\dom}(|\bm|^2-\kappa_0 \bh\cdot\bm) dx ds
	\end{equation*}
	Subtracting this identity from \eqref{eq:limitenergyineq}, we get
	\begin{equation*}
	\frac{1}{2}\int_{\dom}\left(\overline{|\bm(t,x)|^2} -|\bm(t,x)|^2\right)dx \leq  -\frac{1}{\tau}\int_0^t\int_{\dom}\left(\overline{|\bm|^2}-|\bm|^2 - \kappa_0\left(\overline{\bh\cdot \bm}-\bh\cdot\bm\right)\right) dx ds.
	\end{equation*}
	We replace $\overline{\bh\cdot \bm}-\bh\cdot\bm$ using \eqref{eq:mhprod}:
	\begin{equation}\label{eq:whatever}
	\frac{1}{2}\int_{\dom}\left(\overline{|\bm(t,x)|^2} -|\bm(t,x)|^2\right)dx \leq  -\frac{1}{\tau}\int_0^t\int_{\dom}\left(\overline{|\bm|^2}-|\bm|^2 + \kappa_0\left(\overline{|\bh|^2}-|\bh|^2\right)\right) dx ds.
	\end{equation}
	By \cite[Corollary 3.33]{Novotny2004},
	\begin{equation}\label{eq:novotny}
	\int_{G}|\mathbf{v}|^2 dx \leq \liminf_{n\rightarrow\infty}\int_{G}|\mathbf{v}_n|^2 dx\leq \int_{G}\overline{|\mathbf{v}|^2} dx,\quad \text{and}\quad |\mathbf{v}|^2\leq \overline{|\mathbf{v}|^2}\,\text{a.e. on } G, 
	\end{equation}
	$\mathbf{v}\in\{\bm,\bh\}$, $G\in\{\dom,[0,T]\times\dom\}$, therefore  the left hand side of \eqref{eq:whatever} is nonnegative while the right hand side is nonpositive. Hence we have
	\begin{equation*}
	\frac{1}{2}\int_{\dom}\left(\overline{|\bm(t,x)|^2} -|\bm(t,x)|^2\right)dx  = 0.
	\end{equation*}
	Thus, by \eqref{eq:novotny}, $\overline{|\bm(t,x)|^2} = |\bm(t,x)|^2$ for almost every $(t,x)\in[0,T]\times\dom$. By~\cite[Theorem 1.1.1 (iii)]{Evans1990}, this implies that $\bm_n\rightarrow\bm$ strongly in $L^2([0,T]\times\dom)$. Using this in \eqref{eq:mhprod}, we see that also
	\begin{equation*}
	\int_0^T\int_{\dom}|\bh|^2 dxdt = \int_0^T\int_{\dom}\overline{|\bh|^2} dxdt, 
	\end{equation*}
	and thus also $\bh_n\rightarrow\bh$ strongly in $L^2([0,T]\times\dom)$. This allows us to conclude that \eqref{eq:weakstrong} holds (and hence point (iii) in Definition~\ref{def:weaksol}).
	The energy inequality~\eqref{eq:energy0} follows from passing to the limit in the energy inequality for the approximations $\{(\bu_n,\bw_n,\bm_n,\bh_n)\}_n$.
	The weak time continuity of $(\bu,\bw,\bm)$ follows from~\eqref{eq:timederconv} combined with~\cite[Lemma 1.4, Chapter III]{Temam2001}. To see that $\bh$ is weakly continuous in time, we first note that by elliptic regularity (see e.g. \cite[Theorem 6.33]{Folland2001}), since $\phi_t$ solves $\Delta\phi_t = -\Div \bm_t$,
	\begin{equation*}
	\norm{\phi_t}_{W^{p,s+1}}\leq C \norm{\bm_t}_{W^{p,s}},
	\end{equation*}
	for some $s<0$ and hence since $\bh_t = \Grad\phi_t$, we obtain
	\begin{equation*}
	\norm{\bh_t}_{W^{p,s}} \leq C\norm{\bm_t}_{W^{p,s}}\leq C,
	\end{equation*}  
	for some $s<0$ small enough. This combined with $\bh\in L^\infty(0,T;L^2(\dom))$ yields $\bh\in C_w([0,T];L^2(\dom))$ by~\cite[Lemma 1.4, Chapter III]{Temam2001}.
	
\end{proof}

\section{Convergence to the equilibrium $\tau\to 0$}\label{S6}
We now proceed to proving Theorems~\ref{thm:convequilibriumweaksol} and~\ref{thm:convergencerate}. For convenience, we recall the statement of Theorem~\ref{thm:convequilibriumweaksol}:
\begin{theorem}
	Denote $\sole:=(\bue,\bwe,\bme,\bhe)$, $\bhe=\Grad\vfe$ a weak solution of system~\eqref{eq:rosenzweig} for a given $\tau>0$. Then as $\tau\to 0$, a subsequence of $\{\sole\}_{\tau>0}$ converges in $L^2([0,T]\times\dom)$ to a weak solution of~\eqref{eq:limitsys}.  
\end{theorem}
\begin{proof}	
	From the energy balance~\eqref{eq:energy0}, we get that the $\sole=(\bue,\bwe,\bme,\bhe)$ satisfy the following uniform bounds for given $\tau>0$, and fixed time horizon $T>0$,
	\begin{subequations}\label{eq:aprioritaueq}
		\begin{align}
		&\norm{\bue}_{L^\infty(0,T;L^2(\dom))} + \norm{\bue}_{L^2(0,T;H^1(\dom))}\leq C,\label{seq:aprioriueq}\\
		&\norm{\bwe}_{L^\infty(0,T;L^2(\dom))} + \norm{\bwe}_{L^2(0,T;H^1(\dom))}\leq C,\label{seq:aprioriweq}\\
		&\norm{\bhe}_{L^\infty(0,T;L^2(\dom))}\leq C,\label{seq:aprioriheq}\\
		&\norm{\vfe}_{L^\infty(0,T;H^1(\dom))}\leq C, \label{seq:aprioriphieq}\\
		&\norm{\bme}_{L^\infty(0,T;L^2(\dom))}\leq C,\label{seq:apriorimeq} \\
		&\tau^{-1/2}\norm{\bme-\kappa_0\bhe}_{L^2((0,T)\times \dom)}\leq C.\label{seq:apriorimheq}
		\end{align}
	\end{subequations}
	
	The first two bounds~\eqref{seq:aprioriueq} and \eqref{seq:aprioriweq} combined with~\eqref{eq:timederconv}, imply using the Aubin-Lions lemma
	\begin{equation*}
	\bue\rightarrow \bU, \quad \bwe\rightarrow\bW,\quad \text{in } L^p((0,T)\times \dom),\quad 1\leq p<6,
	\end{equation*}
	up to a subsequence, for some limiting functions $\bU,\bW\in L^\infty(0,T;L^2(\dom))\cap L^2(0,T;H^1(\dom))$.
	The third, forth and fifth bound~\eqref{seq:aprioriheq}, \eqref{seq:aprioriphieq} and~\eqref{seq:apriorimheq}, imply, using the Banach-Alaoglu theorem,
	\begin{equation}\label{eq:weakconvHM}
	\begin{split}
	\bme\weakstar\bM,\quad \bhe\weakstar\bH,&\quad \text{ weak* in } L^\infty(0,T;L^2(\dom)),\\
	\vfe\weakstar\Phi,&\quad \text{ weak* in } L^\infty(0,T;H^1(\dom)),
	\end{split}
	\end{equation}
	for limiting functions $\bM,\bH\in L^\infty(0,T;L^2(\dom))$ and $\Phi\in L^\infty(0,T;H^1(\dom))$. Moreover, $\Grad\Phi =\bH$, because $\Grad\vfe=\bhe$ and both quantities converge weak*.
	Using the sixth a priori bound~\eqref{seq:apriorimheq} and combining it with the weak* convergence~\eqref{eq:weakconvHM}, we obtain for any test function $\psi\in L^2([0,T]\times\dom)$, 
	\begin{equation*}
	\begin{split}
	\left|\int_0^T\!\!\int_{\dom}(\kappa_0\bH-\bM)\cdot\psi\, dx \right|&=\left|\lim_{\tau\to 0}\int_0^T\!\!\int_{\dom}(\kappa_0\bhe-\bme)\cdot\psi\, dx\right|\\
	&\leq \lim_{\tau\to 0}\norm{\kappa_0\bhe-\bme}_{L^2}\norm{\psi}_{L^2}\leq C\lim_{\tau\to 0}\tau^{1/2} =0,
	\end{split}
	\end{equation*}
	hence $\kappa_0\bH=\bM$ in $L^2([0,T]\times\dom)$.
	
	Using~\eqref{eq:weakconvHM} to pass to the limit in the weak formulation of equation \eqref{seq:divb}, we get for test functions $\psi\in L^1(0,T;H^1(\dom))$
	\begin{equation}\label{eq:elliptic}
	\int_0^T\int_{\dom}\bh_a\cdot\Grad\psi dx dt=\int_0^T\int_{\dom}(\bhe+\bme)\cdot \Grad\psi dx dt \rightarrow \int_0^T\int_{\dom}(\bH+\bM)\cdot \Grad\psi dx dt,
	\end{equation}
	hence in the limit
	\begin{equation*}
	\int_0^T\!\!\int_{\dom}\bh_a\cdot\Grad\psi dx dt = \int_0^T\int_{\dom}(\bH+\bM)\cdot \Grad\psi dx dt,
	\end{equation*}
	for any test function $\psi\in L^1(0,T;H^1(\dom))$. In particular, we can use $\Phi$ as a test function and get (using $\Grad\Phi = \bH$), 
	\begin{equation}\label{eq:limit1elliptic}
	\int_0^T\!\!\int_{\dom}\bh_a\cdot\Grad\Phi\, dx dt = \int_0^T\int_{\dom}(\bH+\bM)\cdot \bH\, dx dt=\int_0^T\int_{\dom}(1+\kappa_0)|\bH|^2\, dx dt.
	\end{equation}
	On the other hand, we could choose $\psi=\vfe$ as a test function in \eqref{eq:elliptic}, and pass to the limit $\tau\to 0$:
	\begin{equation}\label{eq:limit2elliptic}
	\int_0^T\!\!\int_{\dom}\bh_a\cdot\Grad\Phi\, dx dt  \leftarrow \int_0^T\!\!\int_{\dom}\bh_a\cdot\Grad\vfe\, dx dt  =\int_0^T\!\!\int_{\dom}(\bhe+\bme)\cdot \bhe dx dt \rightarrow \int_0^T\!\!\int_{\dom}\overline{(\bH+\bM)\cdot \bH}\, dx dt.
	\end{equation}
	where $\overline{(\bH+\bM)\cdot\bH}$ is the weak limit of $(\bhe+\bme)\cdot\bhe$. Combining \eqref{eq:limit1elliptic} and \eqref{eq:limit2elliptic}, we get
	\begin{equation}\label{eq:gaga}
	\int_0^T\int_{\dom}\overline{(\bH+\bM)\cdot \bH}\, dx dt :=\lim_{\tau\to 0} \int_0^T\int_{\dom}(\bhe+\bme)\cdot \bhe dx dt  =\int_0^T\int_{\dom}(1+\kappa_0)|\bH|^2 dx dt,
	\end{equation}
	Note that
	\begin{equation*}
	(1+\kappa_0)\norm{\bhe}_{L^2((0,T)\times \dom)}^2=\int_0^T\int_{\dom}\left(|\bhe|^2 +\bhe\cdot\bme\right) dx dt + \int_0^T\int_{\dom}\bhe \left(\kappa_0\bhe-\bme\right) dx dt,
	\end{equation*}
	therefore
	\begin{equation*}
	\begin{split}
	\left|(1+\kappa_0)\norm{\bhe}_{L^2((0,T)\times \dom)}^2- \int_0^T\int_{\dom}\left(|\bhe|^2 +\bhe\cdot\bme\right) dx dt \right|&= \left| \int_0^T\int_{\dom}\bhe \left(\kappa_0\bhe-\bme\right) dx dt\right|\\
	&\leq \norm{\bhe}_{L^2}\norm{\bme-\kappa_0\bhe}_{L^2}\leq C\sqrt{\tau},
	\end{split}
	\end{equation*}
	and
	\begin{align*}
	&\left|(1+\kappa_0)\lim_{\tau\to 0}\norm{\bhe}^2_{L^2((0,T)\times \dom)} -\int_0^T\!\!\int_{\dom}\overline{\bH\cdot(\bH+\bM)} dx dt\right|\\
	&\quad =\bigg|(1+\kappa_0)\lim_{\tau\to 0}\int_0^T\!\!\int_{\dom}|\bhe|^2 dx dt-\lim_{\tau\to 0}\int_0^T\!\!\int_{\dom}\bhe\cdot(\bhe+\bme)dxdt\bigg|\\
	&\quad = \left|\lim_{\tau\to 0}\int_0^T\!\!\int_{\dom}\bhe\cdot(\kappa_0\bhe -\bme) dx dt\right|\\
	&\quad\leq \limsup_{\tau\to 0}\left|\int_0^T\!\!\int_{\dom}\bhe\cdot(\kappa_0\bhe -\bme) dx dt\right|\\
	&\quad\leq \limsup_{\tau\to 0}\int_0^T\!\!\int_{\dom}|\bhe|\left|\kappa_0\bhe -\bme \right|dx dt\\
	&\quad \leq C\limsup_{\tau\to 0}\sqrt{\tau} = 0.
	\end{align*}
	So combining this with~\eqref{eq:gaga}, we get
	\begin{equation*}
	\lim_{\tau\to 0}\norm{\bhe}^2_{L^2((0,T)\times \dom)} = \int_0^T\int_{\dom}|\bH|^2 dx dt.
	\end{equation*}
	Hence
	\begin{align*}
	\lim_{\tau\to 0}\norm{\bhe-\bH}_{L^2}^2 &=\lim_{\tau\to 0}\norm{\bhe}_{L^2}^2- 2 \lim_{\tau\to 0}\int_0^T\int_{\dom}\bhe\cdot\bH dx dt +\norm{\bH}_{L^2}^2\\
	& = - 2 \lim_{\tau\to 0}\int_0^T\int_{\dom}\bhe\cdot\bH dx dt +2\norm{\bH}_{L^2}^2\\
	& = 0,
	\end{align*}
	the last equality follows from the weak convergence of $\{\bhe\}_{\tau> 0}$. Therefore a subsequence of $\{\bhe\}_{\tau> 0}$ converges in fact strongly in $L^2$. Using this and the a priori bound~\eqref{seq:apriorimheq}, we also get strong convergence of a subsequence of $\{\bme\}_{\tau>0}$:
	\begin{equation*}
	\lim_{\tau\to 0}\norm{\bme-\kappa_0\bH}_{L^2}
	\leq \lim_{\tau\to 0}\norm{\bme-\kappa_0\bhe}_{L^2}+\lim_{\tau\to 0}\norm{\kappa_0\bhe-\kappa_0\bH}_{L^2}=0.
	\end{equation*}
	This allows us to pass to the limit in all the terms in the weak formulation of \eqref{eq:rosenzweig} and obtain that a subsequence converges, as $\tau\to 0$, to a weak solution of \eqref{eq:limitsys}. For the last term in \eqref{seq:dtu}, we use Lemma~\ref{lem:integrableweaksol}. In equation \eqref{seq:dtm}, we multiply everything by $\tau$ before passing $\tau\to 0$ in the weak formulation.
\end{proof}
\subsection{Convergence when the solution of the limiting system is smooth}\label{S5}
When the solution of the limiting system~\eqref{eq:limitsys} is smooth, which is the case for short times and smooth data and in two space dimensions for smooth enough data (this is shown in Appendix~\ref{S8}), one can show a convergence rate in $\tau$ for the approximate solutions of~\eqref{eq:rosenzweig}:
\begin{theorem}
	If the solution of~\eqref{eq:limitsys} satisfies $\bU,\bW\in L^\infty(0,T;\mathrm{Lip}(\dom))$, $\bH=\Grad\Phi\in L^2(0,T;H^1(\dom))$, $\partial_t\bH\in L^2([0,T]\times\dom)$; and the initial data for~\eqref{eq:rosenzweig} and~\eqref{eq:limitsys} satisfy
	\begin{equation*}
	\norm{\bu_0-\bU_0}_{L^2(\dom)}^2 +\norm{\bw_0-\bW_0}_{L^2(\dom)}^2+\norm{\bm_0-\bM_0}_{L^2(\dom)}^2 %+\norm{\Div(\bm_0-\bM_0)}_{L^2(\dom)}^2
	\leq C\tau,
	\end{equation*}	
	then the solutions $\sole$ of~\eqref{eq:rosenzweig} converge as $\tau\to 0$ to the solution of the limiting system~\eqref{eq:limitsys} at the rate:
	\begin{multline*}
	\norm{\bue-\bU}_{L^2(\dom)}(t)+\norm{\bwe-\bW}_{L^2(\dom)}(t)+\norm{\bme-\bM}_{L^2(\dom)}(t)\\
	+\norm{\bhe-\bH}_{L^2(\dom)}(t)+\norm{\Grad(\bue-\bU)}_{L^2([0,t]\times\dom)}+\norm{\Grad(\bwe-\bW)}_{L^2([0,t]\times\dom)}\leq C\sqrt{\tau}(1+\exp(Ct)).
	\end{multline*}
	
\end{theorem}
\begin{proof}
	To prove this theorem, we will use the relative entropy method that was introduced by Dafermos~\cite{Dafermos1979,Dafermos1979b} and DiPerna~\cite{DiPerna1979} in the context of hyperbolic systems of conservation laws, see also~\cite{Dafermos1979a}.
	To this end, we define the following relative entropy functional,
	\begin{equation}
	\label{eq:relativeentropydef}
	\mathcal{E}(\sole | \sol_0):= E(\sole)-E(\sol_0)-dE(\sol_0)(\sole-\sol_0),
	\end{equation}
	where the energy $E$ is defined in~\eqref{eq:defE}. The integral of $\mathcal{E}(\sole | \sol_0)$ over $\dom$ measures the distance in $L^2$ of the solution $\sole:=(\bue,\bwe,\bme,\bhe)$ of~\eqref{eq:rosenzweig} from the solution $\sol_0:=(\bU,\bW,\bM,\bH)$ of the limiting system~\eqref{eq:limitsys}. For notational convenience, we will omit writing the dependence of $\bue,\bwe,\bme$ and $\bhe$ on $\tau$ and write $\sole:=(\bu,\bw,\bm,\bh)$, while we denote the solution of the limiting system by $\sol_0:=(\bU,\bW,\bM,\bH)$. $dE(\sol_0)$ is the derivative of $E$ with respect to all variables, that is, $dE(\sol_0)=(\bU,\bW,\frac{\mu_0}{\kappa_0}\bM,\mu_0\bH)$. Some basic algebra shows that $\mathcal{E}(\sole|\sol_0)$ can be written as
	\begin{equation}
	\label{eq:relativeentropy}
	\mathcal{E}(\sole|\sol_0)=\frac{1}{2}\left(|\bu-\bU|^2 + |\bw-\bW|^2 +\frac{\mu_0}{\kappa_0}|\bm-\bM|^2+\mu_0|\bh-\bH|^2\right),
	\end{equation}
	or
	\begin{equation}
	\label{eq:relativeentropydef2}
	\mathcal{E}(\sole | \sol_0):= E(\sole)+E(\sol_0)-dE(\sol_0)\sole,
	\end{equation}
	From the energy (in)equality~\eqref{eq:energy0} and Remark~\ref{rem:energysmoothsol}, we obtain
	\begin{multline}
	\label{eq:energysubtraction}
	\int_{\dom}\left(E(\sole(t))+E(\sol_0(t))\right) dx +\int_0^t\int_{\dom}\left(
	D(\sole)+D(\sol_0)\right)dxds\\
	\leq \int_{\dom}\left(E(\sole(0))+E(\sol_0(0))\right) dx +\mu_0\int_0^t\int_{\dom} \partial_s\bh_a(\bH+\bh) dx ds.
	\end{multline}
	Then we note that, since $\sole\in C_w([0,T];L^2(\dom))$
	\begin{equation*}
	\int_{\dom} dE(\sol_0(t))\sole(t) dx -\int_{\dom} dE(\sol_0(\eta))\sole(\eta) dx = -\lim_{\eps\to 0}\int_0^T\int_{\dom} dE(\sol_0(s))\sole(s)\partial_s\te(s) dx ds,
	\end{equation*}
	where $\te=\mathbf{1}_{[\eta,t]}\ast \omega_\eps$ is a regularized version of the indicator function of the interval $[\eta,t]$, where $\eta\geq\eps>0$ is a small number, $\omega_\eps(s):=\frac{1}{\eps}\omega(s/\eps)$ is a symmetric, nonnegative, smooth, compactly supported on $[-1,1]$ mollifier with $\int\omega(s) ds=1$. Hence, we can compute an expression for $\int_0^T\int_{\dom} dE(\sol_0(s))\sole(s)\partial_s\te(s) dx ds$ using the weak formulations of the equations for $\sole$, see Definition~\ref{def:weaksol}, points (ii) and (iii), and the strong formulation~\eqref{eq:limitsys} for $\sol_0$ since it is assumed to be sufficiently regular.
	We have
	\begin{multline*}
	\int_0^T\int_{\dom} dE(\sol_0(s))\sole(s)\partial_s\te(s) dx ds= \int_0^T\int_{\dom} \left(\bU\bu+\bW\bw+\frac{\mu_0}{\kappa_0}\bM\bm+\kappa_0\bH\bh\right)\partial_s\te(s) dx ds\\
	=\underbrace{\int_0^T\int_{\dom} \bu(\partial_s(\bU\te)-\bU_s\te)dxds}_{(a)} +\underbrace{\int_0^T\int_{\dom}\bw(\partial_s(\bW\te)-\bW_s\te)dx ds}_{(b)}\\+\underbrace{\frac{\mu_0}{\kappa_0}\int_0^T\int_{\dom}\bm(\partial_s(\bM\te)-\bM_s\te)dx ds}_{(c)}+\underbrace{\mu_0\int_0^T\int_{\dom}\bh(\partial_s(\bH\te)-\bH_s\te dx ds}_{(d)}
	\end{multline*}
	Using~\eqref{seq:dtu} and~\eqref{seq:dtuL}, we compute:
	\begin{equation*}
	\begin{split}
	(a)&=\int_0^T\int_{\dom}\left[2\nu\Grad\bu:\Grad\bU+2\nu_r\left(\Curl\bu\Curl\bU-\Curl\bU\bw-\Curl\bu\bW\right)\right] \te dx ds\\
	&\hphantom{=}+\mu_0\int_0^T\int_{\dom}\bh[(\bm+\bh)\cdot\Grad ]\bU \te dx ds\\
	&\hphantom{=}+\int_0^T\int_{\dom} (\bu-\bU)[(\bU-\bu)\cdot\Grad]\bU \te dxds.
	\end{split}
	\end{equation*}
	Equations~\eqref{seq:dtw} and~\eqref{seq:dtwL} yield:
	\begin{equation*}
	\begin{split}
	(b)&=\int_0^T\int_{\dom}\left[2 c_1\Grad\bw:\Grad\bW+2 c_2\Div\bW\Div\bw+8\nu_r\bw\bW -2\nu_r\left(\Curl\bU\bw+\Curl\bu\bW\right)\right] \te dx ds\\
	&\hphantom{=}-\mu_0\int_0^T\int_{\dom}(\bm\times\bh)\cdot\bW\te dx ds+\int_0^T\int_{\dom} (\bw-\bW)[(\bU-\bu)\cdot\Grad]\bW \te dxds.
	\end{split}
	\end{equation*}
	From~\eqref{seq:dtm} and~\eqref{seq:ML}, we obtain
	\begin{equation*}
	\begin{split}
	(c)&=\frac{\mu_0}{\kappa_0}\int_0^T\int_{\dom}\left[-\bm(\bu\cdot\Grad)\bM-(\bw\times\bm)\cdot\bM+\frac{1}{\tau}(\bm-\kappa_0\bh)\bM\right]\te dx ds\\
	&\hphantom{=} -\mu_0\int_0^T\int_{\dom}\bH_s\bm\te dx ds.
	\end{split}
	\end{equation*}
	Finally, using~\eqref{seq:curlh}, \eqref{seq:divb}, \eqref{seq:curlhL} and~\eqref{seq:divbL}, we get,
	\begin{equation*}
	\begin{split}
	(d)&=-\mu_0\int_0^T\int_{\dom}\partial_s\bh_a(\bh+\bH)\te dx ds +\mu_0\int_0^T\int_{\dom}\bm(\bu\cdot\Grad)\bH\te dx ds \\
	&\hphantom{=}+\mu_0\int_0^T\int_{\dom}\bh\bM_s\te dx ds -\frac{\mu_0}{\tau}\int_0^T\int_{\dom}(\bm-\kappa_0\bh)\bH\te dx ds+\mu_0\int_0^T\int_{\dom}(\bw\times\bm)\cdot\bH\te dx ds
	\end{split}
	\end{equation*}
	Combining $(a)-(d)$, sending $\eps\to 0$ and using $\bM=\kappa_0\bH$ (equation~\eqref{seq:ML}), we obtain
	\begin{align}\label{eq:crosstermRE}
	\begin{split}
	\int_{\dom}( dE(\sol_0(t))\sole(t) - dE(\sol_0(\eta))\sole(\eta)) dx& = -2\int_\eta^t\int_{\dom}\left[\nu\Grad\bu:\Grad\bU + c_1\Grad\bw:\Grad\bW+ c_2\Div\bW\Div\bw\right] dx ds\\
	&\hphantom{=}-\nu_r\int_\eta^t\int_{\dom}\left(2\Curl\bu\Curl\bU-4\Curl\bU\bw-4\Curl\bu\bW+8\bw\bW\right) dx ds\\
	&\hphantom{=}-\mu_0\int_\eta^t\int_{\dom}\bh[(\bm+\bh)\cdot\Grad ]\bU  dx ds\\
	&\hphantom{=}-\int_\eta^t\int_{\dom} (\bu-\bU)[(\bU-\bu)\cdot\Grad]\bU  dxds\\
	&\hphantom{=}-\int_\eta^t\int_{\dom} (\bw-\bW)[(\bU-\bu)\cdot\Grad]\bW  dxds\\
	&\hphantom{=}+\mu_0\int_\eta^t\int_{\dom}(\bm\times\bh)\cdot\bW dx ds%\\
	+\mu_0\int_\eta^t\int_{\dom}\bH_s(\bm-\kappa_0\bh) dx ds\\
	&\hphantom{=}+\mu_0\int_\eta^t\int_{\dom}\partial_s \bh_a(\bh+\bH) dx ds. 
	\end{split}
	\end{align}
	This holds for any $\eta>0$. Using the weak continuity of $\sole$ and $\sol_0$ and the integrability properties, we can pass $\eta\to 0$. Subtracting the result from~\eqref{eq:energysubtraction}, we get
	\begin{multline*}
	\int_{\dom}\mathcal{E}(\sole|\sol_0)(t) dx +\int_0^t\int_{\dom} \nu|\Grad(\bu-\bU)|^2 dx ds+\frac{\mu_0}{\kappa_0\tau}\int_0^t\int_{\dom}|\bm-\kappa_0\bh|^2 dx ds\\
	+\int_0^t\int_{\dom} \left(c_1|\Grad(\bw-\bW)|^2 + c_2 |\Div(\bw-\bW)|^2 +\nu_r|2(\bw-\bW)-\Curl(\bu-\bU)|^2\right) dx ds 
	\\\leq \int_{\dom}\mathcal{E}(\sole|\sol_0)(0) dx-\mu_0\int_0^t\int_{\dom}\bH_s(\bm-\kappa_0\bh) dx ds \\
	+\int_0^t\int_{\dom} (\bu-\bU)[(\bU-\bu)\cdot\Grad]\bU  dxds
	+\int_0^t\int_{\dom} (\bw-\bW)[(\bU-\bu)\cdot\Grad]\bW  dxds\\
	+\mu_0\int_0^t\int_{\dom}\bh[(\bm+\bh)\cdot\Grad ]\bU  dx ds-\mu_0\int_0^t\int_{\dom}(\bm\times\bh)\cdot\bW dx ds.
	\end{multline*}
	Let us denote
	\begin{multline*}
	\mathcal{D}(\sole|\sol_0):=\nu|\Grad(\bu-\bU)|^2+\frac{\mu_0}{\kappa_0\tau}|\bm-\kappa_0\bh|^2
	+c_1|\Grad(\bw-\bW)|^2 \\+ c_2 |\Div(\bw-\bW)|^2 +\nu_r|2(\bw-\bW)-\Curl(\bu-\bU)|^2. 
	\end{multline*}
	Then the previous identity becomes
	\begin{multline}\label{eq:relentropyapriori}
	\int_{\dom}\mathcal{E}(\sole|\sol_0)(t) dx +\int_0^t\int_{\dom} \mathcal{D}(\sole|\sol_0)(s) dx ds 
	\\ \leq \int_{\dom}\mathcal{E}(\sole|\sol_0)(0) dx+\underbrace{\int_0^t\int_{\dom} (\bu-\bU)[(\bU-\bu)\cdot\Grad]\bU  dxds}_{\text{I}_1} \\
	+\underbrace{\int_0^t\int_{\dom} (\bw-\bW)[(\bU-\bu)\cdot\Grad]\bW  dxds}_{\text{I}_2}-\underbrace{\mu_0\int_0^t\int_{\dom}\bH_s(\bm-\kappa_0\bh) dx ds}_{\text{II}}\\
	+\underbrace{\mu_0\int_0^t\int_{\dom} \bh[(\bm+\bh)\cdot\Grad ]\bU  dx ds}_{\text{III}}-\underbrace{\mu_0\int_0^t\int_{\dom}(\bm\times\bh)\cdot\bW dx ds}_{\text{IV}}.%\\
	\end{multline}
	We start by bounding the terms $\text{I}_1$ and $\text{I}_2$.
	\begin{equation*}
	|\mathrm{I}_1|\leq \int_0^t\int_{\dom}|\bu-\bU|^2 |\Grad\bU| dxds \leq \norm{\Grad\bU}_{L^\infty}\norm{\bu-\bU}_{L^2([0,t]\times\dom)}^2,
	\end{equation*}
	and
	\begin{equation*}
	\begin{split}
	|\mathrm{I}_2|\leq& \int_0^t\int_{\dom}|\bu-\bU| |\bw-\bW| |\Grad\bW| dxds \leq \norm{\Grad\bW}_{L^\infty}\norm{\bu-\bU}_{L^2([0,t]\times\dom)}\norm{\bw-\bW}_{L^2([0,t]\times\dom)}\\
	&\leq \frac{1}{2}\norm{\Grad\bW}_{L^\infty}\left(\norm{\bu-\bU}_{L^2([0,t]\times\dom)}^2+\norm{\bw-\bW}^2_{L^2([0,t]\times\dom)}\right).
	\end{split}
	\end{equation*}
	Using Young's inequality and the regularity of the limit functions $(\bU,\bW,\bH)$, we bound term II as follows (notice that it indeed follows that $\bH_t\in L^2(\dom)$ as long as $\partial_t\bh_a\in L^2(\dom)$):
	\begin{equation*}
	\begin{split}
	|\text{II}|&\leq \mu_0\int_0^t\int_{\dom}|\bm-\kappa_0\bh| |\bH_s| dxds \\
	&\leq \frac{\mu_0}{8\kappa_0\tau}\int_0^t\int_{\dom}|\bm-\kappa_0\bh|^2 dxds +4\tau\mu_0\kappa_0\int_0^t\int_{\dom}|\bH_s|^2 dxds.
	\end{split}
	\end{equation*}
	The second term on the right hand side is of order $C\tau$ since $\bH_t\in L^2(\dom)$.
	
	Assuming for the moment that all the involved functions are smooth, we rewrite term $\text{III}$ as follows (the general case follows by approximation, for example one can mollify all the variables and because the magnetostatic equation~\eqref{seq:divb}/\eqref{seq:divbL} is linear, the mollified variables solve a mollified equation):
	\begin{align*}
	\text{III}& = \mu_0\int_0^t\int_{\dom}\left(\bh\left[(\bm+\bh)\cdot\Grad\right]\bU\right)dx ds\\
	& = \mu_0\int_0^t\int_{\dom}\big[\bh\cdot\left(\left[(\bm-\kappa_0\bh)\cdot\Grad\right]\bU\right) +(1+\kappa_0)\bh\cdot\left([\bh\cdot\Grad]\bU\right)\bigr] dx ds\\
	& = \mu_0\int_0^t\int_{\dom}(\bh-\bH)\cdot\left(\left[(\bm-\kappa_0\bh)\cdot\Grad\right]\bU\right)dxds+\mu_0\int_0^t\int_{\dom}\bH\cdot\left(\left[(\bm-\kappa_0\bh)\cdot\Grad\right]\bU\right)dxds\\
	&\hphantom{=} +\mu_0(1+\kappa_0)\int_0^t\int_{\dom}\left([\bh\cdot\Grad](\bU\bh)\right) dx ds\\
	& = \mu_0\int_0^t\int_{\dom}(\bh-\bH)\cdot\left(\left[(\bm-\kappa_0\bh)\cdot\Grad\right]\bU\right)dxds-\mu_0\int_0^t\int_{\dom}\bU\cdot\left(\left[(\bm-\kappa_0\bh)\cdot\Grad\right]\bH\right)dxds\\
	&\hphantom{=}+\mu_0\int_0^t\int_{\dom}\left[(\bm-\kappa_0\bh)\cdot\Grad\right](\bH\bU)dxds  +\mu_0(1+\kappa_0)\int_0^t\int_{\dom}\left([\bh\cdot\Grad](\bU\bh)\right) dx ds\\
	& = \mu_0\int_0^t\int_{\dom}(\bh-\bH)\cdot\left(\left[(\bm-\kappa_0\bh)\cdot\Grad\right]\bU\right)dxds-\mu_0\int_0^t\int_{\dom}\bU\cdot\left(\left[(\bm-\kappa_0\bh)\cdot\Grad\right]\bH\right)dxds\\
	&\hphantom{=}+\mu_0\int_0^t\int_{\dom} \bh_a\cdot\Grad(\bH\bU) dxds -\mu_0(1+\kappa_0)\int_0^t\int_{\dom}\left[\bh\cdot\Grad\right](\bH\bU)dxds \\
	&\hphantom{=} +\mu_0(1+\kappa_0)\int_0^t\int_{\dom}\left([\bh\cdot\Grad](\bU\bh)\right) dx ds\\
	& = \mu_0\int_0^t\int_{\dom}(\bh-\bH)\cdot\left(\left[(\bm-\kappa_0\bh)\cdot\Grad\right]\bU\right)dxds-\mu_0\int_0^t\int_{\dom}\bU\cdot\left(\left[(\bm-\kappa_0\bh)\cdot\Grad\right]\bH\right)dxds\\
	&\hphantom{=}	+\mu_0(1+\kappa_0)\int_0^t\int_{\dom}\left([\bh\cdot\Grad](\bU(\bh-\bH))\right) dx ds\\
	&=\mu_0\int_0^t\int_{\dom}(\bh-\bH)\cdot\left(\left[(\bm-\kappa_0\bh)\cdot\Grad\right]\bU\right)dxds-\mu_0\int_0^t\int_{\dom}\bU\cdot\left(\left[(\bm-\kappa_0\bh)\cdot\Grad\right]\bH\right)dxds\\
	&\hphantom{=} +\mu_0(1+\kappa_0)\int_0^t\int_{\dom}\left([(\bh-\bH)\cdot\Grad](\bU(\bh-\bH))\right) dx ds\\
	&=\mu_0\int_0^t\int_{\dom}(\bh-\bH)\cdot\left(\left[(\bm-\kappa_0\bh)\cdot\Grad\right]\bU\right)dxds-\mu_0\int_0^t\int_{\dom}\bU\left(\left[(\bm-\kappa_0\bh)\cdot\Grad\right]\bH\right)dxds\\
	&\hphantom{=} +\mu_0(1+\kappa_0)\int_0^t\int_{\dom}(\bh-\bH)\cdot\left([(\bh-\bH)\cdot\Grad]\bU\right) dx ds
	\end{align*}
	where we also used in the third equality that, since $\bh=\Grad\varphi$ and $\Div\bU=0$,
	\begin{equation}\label{eq:waek}
	\begin{split}
	&\int_{\dom}\bU\cdot [(\bh\cdot\Grad)\bh] dx = \int_{\dom}\bU\cdot [(\Grad\varphi\cdot\Grad)\Grad\varphi] dx =\sum_{i,j=1}^{3}\bU^{(i)}\partial_j\varphi\partial_j\partial_i\varphi dx\\ 
	&\quad=\frac{1}{2}\sum_{i,j=1}^{3}\bU^{(i)}\partial_i|\partial_j\varphi|^2 dx=-\int_{\dom}\Div\bU |\Grad\varphi|^2 dx =0.
	\end{split}
	\end{equation}
	and equation~\eqref{seq:divbL} in the second last equality, and~\eqref{eq:waek} with $\bh$ replaced by $(\bh-\bH)$ in the last equality. Hence
	\begin{multline}
	\label{eq:faenx}
	\text{III} = \mu_0\int_0^t\int_{\dom}\bh\left[(\bm+\bh)\cdot\Grad\right]\bU dx ds\\
	=\underbrace{\mu_0\int_0^t\int_{\dom}(\bh-\bH)\cdot\left(\left[(\bm-\kappa_0\bh)\cdot\Grad\right]\bU\right)dxds}_{\text{III}_1}-\underbrace{\mu_0\int_0^t\int_{\dom}\bU\cdot\left(\left[(\bm-\kappa_0\bh)\cdot\Grad\right]\bH\right)dxds}_{\text{III}_2}\\
	+\underbrace{\mu_0(1+\kappa_0)\int_0^t\int_{\dom}(\bh-\bH)\cdot\left([(\bh-\bH)\cdot\Grad]\bU\right) dx ds}_{\text{III}_3}
	\end{multline}
	We observe that in this last identity~\eqref{eq:faenx} all terms are bounded when $\bh,\bm\in L^2([0,t]\times\dom)$ and $\bU,\bH\in L^\infty(0,T;\text{Lip}(\dom))$ and hence it holds in our situation by approximation of all quantities with smooth functions.
	We continue to bound the terms $\text{III}_1$, $\text{III}_2$ and $\text{III}_3$:
	For the first term, $\text{III}_1$, we use Young's inequality and then H\"{o}lder's inequality:
	\begin{equation*}
	|\text{III}_{1}|\leq \frac{\mu_0}{8\kappa_0\tau}\int_0^t\int_{\dom}|\bm-\kappa_0\bh|^2 dxds +2\mu_0\kappa_0\tau\norm{\Grad \bU}_{L^\infty}^2 \int_0^t\int_{\dom}|\bh-\bH|^2 dxds.
	\end{equation*}
	Using Young's inequality, we also bound the second term, $\text{III}_2$ by
	\begin{equation*}
	|\text{III}_2|\leq \frac{\mu_0}{8\kappa_0\tau}\int_0^t\int_{\dom}|\bm-\kappa_0\bh|^2 dxds +2\tau\mu_0\kappa_0\int_0^t\int_{\dom}|\bU|^2|\Grad\bH|^2 dx ds.
	\end{equation*}
	For the third term, we use again H\"{o}lder's inequlity
	\begin{equation*}
	|\text{III}_3|\leq \mu_0(1+\kappa_0)\norm{\Grad \bU}_{L^\infty}\int_0^t\int_{\dom}|\bh-\bH|^2 dx ds.
	\end{equation*}
	Combining the three, we get
	\begin{multline}\label{eq:boundIII}
	|\text{III}|\leq \frac{\mu_0}{4\kappa_0\tau}\int_0^t\int_{\dom}|\bm-\kappa_0\bh|^2 dxds+2\tau\mu_0\kappa_0\int_0^t\int_{\dom}|\bU|^2|\Grad\bH|^2 dx ds\\ +\mu_0\norm{\Grad \bU}_{L^\infty}\left(2\kappa_0\tau\norm{\Grad \bU}_{L^\infty}+1+\kappa_0\right) \int_0^t\int_{\dom}|\bh-\bH|^2 dxds. 
	\end{multline}  
	To bound term IV, we note that since $\bh\times\bh=0$, we can rewrite it as 
	\begin{equation*}
	\text{IV}= -\mu_0\int_0^t\int_{\dom}\bW\cdot \left(\left(\bm-{\kappa_0}\bh\right)\times\bh\right) dxds.
	\end{equation*}
	Using H\"older and Young's inequality, we can bound it as follows:
	\begin{align*}
	|\text{IV}|&\leq \mu_0 \norm{\bW}_{L^\infty} \norm{\bm-\kappa_0\bh}_{L^2([0,t]\times\dom)}\norm{\bh}_{L^2([0,t]\times\dom)}\\
	&\leq \frac{\mu_0}{8\tau\kappa_0}\norm{\bm-\kappa_0\bh}_{L^2([0,t]\times\dom)}^2 + 2\mu_0\tau\kappa_0\norm{\bh}_{L^2([0,t]\times\dom)}^2\norm{\bW}^2_{L^\infty}.
	\end{align*}
	Thus, under the assumption that $\bU,\bW\in L^\infty(0,T;\text{Lip}(\dom))$ and $\bH\in L^2(0,T;H^1(\dom))$, $\partial_t\bH\in L^2([0,T]\times\dom)$, the evolution of the relative entropy,~\eqref{eq:relentropyapriori}, is bounded as follows:
	\begin{multline}\label{eq:relentropybound}
	\int_{\dom}\mathcal{E}(\sole|\sol_0)(t) dx +\int_0^t\int_{\dom} \mathcal{D}(\sole|\sol_0)(s) dx ds 
	\\
	\leq \int_{\dom}\mathcal{E}(\sole|\sol_0)(0) dx+C\int_0^t\int_{\dom}|\bu-\bU|^2 dxds + C\int_0^t\int_{\dom}|\bw-\bW|^2 dx ds \\
	+ C \int_0^t\int_{\dom}|\bh-\bH|^2 dxds 
	+\frac{\mu_0}{2\kappa_0\tau}\int_0^t\int_{\dom}|\bm-\kappa_0\bh|^2 dxds + C\tau.
	\end{multline}
	Defining $\widetilde{\mathcal{D}}(\sole|\sol_0)$ by
	\begin{multline*}
	\widetilde{\mathcal{D}}(\sole|\sol_0):=\nu|\Grad(\bu-\bU)|^2+\frac{\mu_0}{2\kappa_0\tau}|\bm-\kappa_0\bh|^2
	+c_1|\Grad(\bw-\bW)|^2 \\+ c_2 |\Div(\bw-\bW)|^2 +\nu_r|2(\bw-\bW)-\Curl(\bu-\bU)|^2, 
	\end{multline*}
	this becomes
	\begin{multline}\label{eq:relentropybound2}
	\int_{\dom}\mathcal{E}(\sole|\sol_0)(t) dx +\int_0^t\int_{\dom} \widetilde{\mathcal{D}}(\sole|\sol_0)(s) dx ds 
	\\
	\leq \int_{\dom}\mathcal{E}(\sole|\sol_0)(0) dx+C\int_0^t\int_{\dom} \mathcal{E}(\sole|\sol_0)(s) dx ds + C\tau.
	\end{multline}
	Now using Gr\"{o}nwall's inequality for $A(t):=\int_{\dom}\mathcal{E}(\sole|\sol_0)(t) dx$, we obtain
	\begin{equation*}
	\int_{\dom}\mathcal{E}(\sole|\sol_0)(t) dx\leq \left(\int_{\dom}\mathcal{E}(\sole|\sol_0)(0) dx+C \tau\right)\exp(Ct).
	\end{equation*}
	Using this in~\eqref{eq:relentropybound2}, we can also bound $\widetilde{\mathcal{D}}(\sole|\sol_0)$:
	\begin{equation*}
	\int_0^t\int_{\dom} \widetilde{\mathcal{D}}(\sole|\sol_0)(s) dx ds \leq \left(\int_{\dom}\mathcal{E}(\sole|\sol_0)(0) + C\tau\right)(C+\exp(Ct)).
	\end{equation*}
	Therefore, if the initial data satisfy
	\begin{equation*}
	\norm{\bu_0-\bU_0}_{L^2(\dom)}^2 +\norm{\bw_0-\bW_0}_{L^2(\dom)}^2+\norm{\bm_0-\bM_0}_{L^2(\dom)}^2 \leq C\tau,
	\end{equation*}
	then, since
	\begin{equation*}
	\Delta(\varphi_0-\Phi_0)=-\Div(\bm_0-\bM_0),
	\end{equation*}
	and hence by the Lax-Milgram theorem,
	\begin{equation*}
	\norm{\bh_0-\bH_0}_{L^2(\dom)}=\norm{\Grad\varphi_0-\Grad\Phi_0}_{L^2(\dom)}\leq\norm{\bm_0-\bM_0}_{L^2(\dom)}\leq C\sqrt{\tau}, 
	\end{equation*}
	we have that
	\begin{equation*}
	\int_{\dom}\mathcal{E}(\sole|\sol_0)(0) dx \leq C\tau.
	\end{equation*}
	Therefore
	\begin{equation*}
	\int_{\dom}\mathcal{E}(\sole|\sol_0)(t) dx +\int_0^t\int_{\dom}\mathcal{D}(\sole|\sol_0)(s) dx ds \leq C\tau (1+\exp(Ct)),
	\end{equation*}
	which proves Theorem~\ref{thm:convergencerate}.
\end{proof}

\section{Acknowledgements}
The work of K.T.  was supported  in part by the National Science Foundation under the grant DMS-161496. The work of F.W. was supported by the Norwegian Research Council, project number  250302. F.W. would like to thank The Offspring for their inspiring work that was helpful for the completion of this manuscript.

\appendix

\section{Regular solutions of limiting system in 2d}\label{S8}
The goal of this section is to apply the arguments from the proof of regularity and uniqueness of the 2D Navier-Stokes equations, see for example~\cite{Temam2001}, to system~\eqref{eq:limitsys} to show it has a unique regular solution if the spatial dimension is 2 and the initial data and $f$ are sufficiently smooth. First we notice that from~\eqref{seq:ML}, equations~\eqref{seq:curlhL}--\eqref{seq:divbL} become
\begin{equation*}
\Delta\Phi(t,x) =0,\quad (t,x)\in [0,T]\times\dom;\quad \frac{\partial\Phi}{\partial n}=\frac{1}{1+\kappa_0}\bh_a\cdot\normal,\quad (t,x)\in [0,T]\times\partial\dom;\quad \int_{\dom}\Phi \,dx =0,
\end{equation*}
and so $\Phi$ is completely decoupled from $\bU$ and $\bW$. Hence if $\bh_a$ and the domain are sufficiently smooth, then by the trace theorem (see e.g. \cite[Thm 2.5.3]{Nedelec2001}) and elliptic regularity, the unique $\Phi$ is also smooth.
Moreover, the source term in~\eqref{seq:dtuL} can be rewritten as
\begin{equation*}
\mu_0(\bM\cdot\Grad)\bH = \frac{\mu_0\kappa_0}{2}\Grad\left(|\bH|^2\right),
\end{equation*}
and added to the pressure. Hence it remains to show that the solution $(\bU,\bW)$ of the system
\begin{subequations}
	\label{eq:simplified2d}
	\begin{align}
	\bU_t +(\bU\cdot\Grad) \bU- (\nu+\nu_r)\Delta \bU+\Grad \widetilde{P} &= 2 \nu_r \Curl \bW ,\label{seq:U2d}\\
	\Div \bU &=0\label{seq:divu2d},\\
	\bW_t + (\bU\cdot \Grad)\bW - c_1\Delta \bW -c_2 \Grad\Div\bW +4 \nu_r \bW &= 2\nu_r \Curl \bU,\label{seq:W2d}
	\end{align}
\end{subequations}
where $\widetilde{P}=P-\frac{1}{2}\mu_0\kappa_0|\bH|^2 $, is regular, which is very similar to showing regularity and uniqueness of the 2D Navier-Stokes equations. We add an outline of the proof here for completeness and because there is some coupling between the two equations involved. 
\begin{theorem}[Uniqueness of solutions in 2d]
	Let the spatial dimension $d=2$. Then the solution $(\bU,\bW)$ of \eqref{eq:simplified2d} with initial data $\bU_0\in \hdiv$ and $\bW_0\in L^2(\dom)$ is unique and continuous as a function from $[0,T]$ into $L^2(\dom)$.	
\end{theorem}
\begin{proof}
	The regularity follows as in the proof of \cite[Theorem III.3.2]{Temam2001}: From the equations~\eqref{eq:simplified2d} and Lemma III.3.4 in~\cite{Temam2001}, it follows that $\partial_{t}\bU, \partial_{t}\bW\in L^2(0,T;H^{-1}(\dom))$. Then~\cite[Lemma III.1.2]{Temam2001} implies that $\bU$ and $\bW$ are almost everywhere equal to continuous functions from $[0,T]$ to $\hdiv$ and $L^2(\dom)$ respectively, that is, $\bU\in C([0,T];\hdiv)$, $\bW\in C([0,T];L^2(\dom))$. Moreover,
	\begin{equation}\label{eq:magic}
	\frac{d}{dt}\norm{\bU(t)}_{L^2}^2 = 2\langle \partial_{t}\bU,\bU\rangle,\quad \frac{d}{dt}\norm{\bW(t)}_{L^2}^2 = 2\langle \partial_{t}\bW,\bW\rangle,
	\end{equation}
	where $\langle\cdot,\cdot,\rangle$ denotes the duality pairing between $H^1_0$ and its dual $H^{-1}$. 
	We assume that $(\bU_1,\bW_1)$ and $(\bU_2,\bW_2)$ are two solutions of~\eqref{eq:simplified2d}. Then the difference $(\bU,\bW):=(\bU_1-\bU_2,\bW_1-\bW_2)$ satisfies
	\begin{align}\label{eq:diffeq}
	\begin{split}
	\bU_t +(\bU_1\cdot\Grad) \bU_1-(\bU_2\cdot\Grad) \bU_2- (\nu+\nu_r)\Delta \bU+\Grad {P} &= 2 \nu_r \Curl \bW ,\\
	\Div \bU &=0,\\
	\bW_t + (\bU_1\cdot \Grad)\bW_1-(\bU_2\cdot \Grad)\bW_2 - c_1\Delta \bW -c_2 \Grad\Div\bW +4 \nu_r \bW &= 2\nu_r \Curl \bU,
	\end{split}
	\end{align}
	with initial condition $\bU(0,\cdot)=0$ and $\bW(0,\cdot)=0$ and $P:=\widetilde{P}_1-\widetilde{P}_2$. 
	We take the a.e. in $t$ the inner product of these equations with $\bU$ and $\bW$ respectively and use~\eqref{eq:magic} to obtain
	\begin{multline*}
	\frac{d}{dt}\left(\norm{\bU(t)}_{L^2}^2+\norm{\bW(t)}_{L^2}^2\right)+2\nu\norm{\Grad\bU}_{L^2}^2+2 c_1\norm{\Grad\bW}_{L^2}^2+2c_2\norm{\Div\bW}_{L^2}^2+2\nu_r\norm{\Curl\bU-2\bW}_{L^2}^2\\
	=-2\int_{\dom}\left[(\bU_1\cdot\Grad)\bU_1-(\bU_2\cdot\Grad)\bU_2\right]\cdot\bU\, dx -2\int_{\dom}\left[(\bU_1\cdot\Grad)\bW_1-(\bU_2\cdot\Grad)\bW_2\right]\cdot\bW\, dx\\
	=-2\int_{\dom}\left[(\bU\cdot\Grad)\bU_2\right]\cdot\bU\, dx -2\int_{\dom}\left[(\bU\cdot\Grad)\bW_2\right]\cdot\bW\, dx.
	\end{multline*}
	We bound the two terms on the right hand side using Ladyshenskaya's inequality (see e.g. Lemma~III.3.3 in~\cite{Temam2001}) and Young's inequality:
	\begin{align*}
	\left|2\int_{\dom}\left[(\bU\cdot\Grad)\bU_2\right]\cdot\bU\, dx\right|&\leq C\norm{\bU}_{L^4}^2\norm{\Grad\bU_2}_{L^2}\\
	&\leq C \norm{\bU}_{L^2}\norm{\Grad\bU}_{L^2}\norm{\Grad\bU_2}_{L^2}\\
	&\leq \frac{\nu}{2}\norm{\Grad\bU}_{L^2}^2+\frac{C}{\nu}\norm{\bU}_{L^2}^2\norm{\Grad\bU_2}_{L^2}^2.
	\end{align*}
	Similarly,
	\begin{align*}
	\left|2\int_{\dom}\left[(\bU\cdot\Grad)\bW_2\right]\cdot\bW\, dx\right|&\leq C\norm{\bW}_{L^4}\norm{\bU}_{L^4}\norm{\Grad\bW_2}_{L^2}\\
	&\leq C \norm{\bU}_{L^2}^{1/2}\norm{\Grad\bU}_{L^2}^{1/2}\norm{\bW}_{L^2}^{1/2}\norm{\Grad\bW}_{L^2}^{1/2}\norm{\Grad\bW_2}_{L^2}\\
	&\leq \frac{\nu}{2}\norm{\Grad\bU}_{L^2}^2+\frac{C}{\nu}\norm{\bU}_{L^2}^2\norm{\Grad\bW_2}_{L^2}^2+c_1\norm{\Grad\bW}_{L^2}^2+\frac{C}{c_1}\norm{\bW}_{L^2}^2\norm{\Grad\bW_2}_{L^2}^2,
	\end{align*}
	and hence
	\begin{multline*}
	\frac{d}{dt}\left(\norm{\bU(t)}_{L^2}^2+\norm{\bW(t)}_{L^2}^2\right)+\nu\norm{\Grad\bU}_{L^2}^2+ c_1\norm{\Grad\bW}_{L^2}^2+2c_2\norm{\Div\bW}_{L^2}^2+2\nu_r\norm{\Curl\bU-2\bW}_{L^2}^2\\
	\leq \frac{C}{\nu}\norm{\bU}_{L^2}^2\left(\norm{\Grad\bU_2}_{L^2}^2+
	\norm{\Grad\bW_2}_{L^2}^2\right)+\frac{C}{c_1}\norm{\bW}_{L^2}^2\norm{\Grad\bW_2}_{L^2}^2.
	\end{multline*}
	Using that $\norm{\Grad\bU_2}_{L^2}^2, \norm{\Grad\bW_2}_{L^2}^2\in L^1([0,T])$ and then applying Gr\"{o}nwall's inequality, we obtain
	\begin{equation*}
	\norm{\bU(t)}_{L^2}^2+\norm{\bW(t)}_{L^2}^2\leq \left(\norm{\bU(0)}_{L^2}^2+\norm{\bW(0)}_{L^2}^2\right)\exp\left(C\int_0^t  (\norm{\Grad\bU_2(s)}_{L^2}^2+ \norm{\Grad\bW_2(s)}_{L^2}^2) ds \right),
	\end{equation*}
	and hence since $\bU(0)=\bW(0)=0$, that $\bU(t)=\bW(t)=0$.
\end{proof}
%%%%%%%%%%%%%%%%%%%%%%%%%%%%%%%%%%%%%%%%%%%%%%%%%%
%%%%%%%%%%%%%%%%%%%%%%%%%%%%%%%%%%%%%%%%%%%%%%%%%%
To show the regularity of the solutions, we will consider a Galerkin approximations of the functions $\bU$ and $\bW$, show uniform estimates in terms of the number of basis functions and then pass to the limit in the approximation to see the same holds for the limiting functions. Therefore, let $\{a_j\}_{j=1}^\infty\subset\hdiv$ be a smooth basis of orthogonal eigenfunctions of the Stokes operator with eigenvalues $\{\lambda_j\}_{j=1}^\infty$, satisfying $\dots\geq \lambda_{j+1}\geq \lambda_j\geq \lambda_{j-1}\geq \dots\geq \lambda_1> 0$ and
\begin{equation}\label{eq:orthogonal1}
\int_{\dom}a_i\cdot a_j dx =\delta_{ij},\quad \nu\int_{\dom}\Grad a_i:\Grad a_j dx =\lambda_i \delta_{ij},
\end{equation}
where $\delta_{ij}$ is the Kronecker delta. In addition, let $\{d_j\}_{j=1}^\infty$ be a smooth basis of orthogonal eigenfunctions of the operator $(c_1\Delta+c_2\Grad\Div)$ in $H^1_0(\dom)$ satisfying
\begin{equation}\label{eq:orthogonal2}
\int_{\dom}d_i\cdot d_j dx =\delta_{ij}, \quad \int_{\dom}\left(c_1\Grad d_i:\Grad d_j +c_2\Div d_i\Div d_j\right)dx = \xi_i\delta_{ij},
\end{equation}
where $\{\xi_j\}_{j=1}^\infty$ are the eigenvalues of the operator.
We denote by $\bUN$ and $\bWN$ the functions
\begin{equation*}
\bUN=\sum_{j=1}^N\ajn(t) a_j,\quad \bWN = \sum_{j=1}^N\bjn(t) d_j,
\end{equation*} 
where $\ajn$ and $\bjn$ satisfy the equations, for $j=1,\dots,N$,
\begin{equation}\label{eq:galerkinapprox}
\begin{split}
&\frac{d}{dt}\int_{\dom}\bUN\cdot a_j dx +\int_{\dom}(\bUN\cdot\Grad)\bUN\cdot a_j dx +(\nu+\nu_r)\int_{\dom}\Grad\bUN:\Grad a_j dx 
= 2\nu_r\int_{\dom}\Curl\bWN \cdot a_j dx\\
&\frac{d}{dt}\int_{\dom}\bWN\cdot d_j dx + \int_{\dom}(\bUN\cdot\Grad )\bWN \cdot d_j dx +c_1\int_{\dom}\Grad \bWN:\Grad d_j dx + c_2\int_{\dom}\Div\bWN\Div d_j dx  \\
& \qquad=-4\nu_r\int_{\dom}\bWN\cdot d_ j dx+ 2\nu_r \int_{\dom}\Curl\bUN\cdot d_j dx.
\end{split}
\end{equation}
Using~\eqref{eq:orthogonal1} and~\eqref{eq:orthogonal2}, this can be simplified to
\begin{align}
\begin{split}\label{eq:odesys1}
&\frac{d}{dt}\ajn(t) +\sum_{k,\ell=1}^N\alpha_{k,N}(t)\alpha_{\ell,N}(t)\int_{\dom}(a_k\cdot\Grad)a_{\ell}\cdot a_j dx +\left(1+\frac{\nu_r}{\nu}\right)\lambda_j \ajn(t) \\
&\qquad = 2\nu_r\sum_{k=1}^N\beta_{k,N}(t)\int_{\dom}\Curl d_k \cdot a_j dx
\end{split}\\
\begin{split}\label{eq:odesys2}
&\frac{d}{dt}\bjn(t) + \sum_{k,\ell=1}^N\alpha_{k,N}(t)\beta_{\ell,N}(t)\int_{\dom}(a_k\cdot\Grad )d_\ell \cdot d_j dx +\xi_j \bjn(t) +4\nu_r \bjn(t)	\\
& \qquad= 2\nu_r \sum_{k=1}^N \alpha_{k,N}(t) \int_{\dom}\Curl a_k\cdot d_j dx.
\end{split}
\end{align} 
Equations~\eqref{eq:odesys1}--\eqref{eq:odesys2} is a nonlinear, locally Lipschitz, system of ODEs with initial data
\begin{equation*}
\alpha_{j,N}^0 =\int_{\dom}\bU_0 \cdot a_j dx,\quad \beta_{j,N}^0 = \int_{\dom}\bW_0\cdot d_j dx,
\end{equation*} 
and therefore a solution exists on some time interval $[0,t_N]$. The existence on the full time interval $[0,T]$ follows as in the case of the Navier-Stokes equations by deriving a uniform in time bound on the $\ajn$ and $\bjn$ which results in an energy inequality for $\bUN$ and $\bWN$ and yields $\bUN,\bWN\in L^\infty(0,T;L^2(\dom))\cap L^2(0,T;H^1_0(\dom))$ uniformly in $N$. Passing to the limit $N\to\infty$ and using the uniform apriori estimates, one concludes existence of limiting functions $\bU$ and $\bW$ solving~\eqref{eq:simplified2d}.

The regularity follows from the following lemmas:
\begin{lemma}
	\label{lem:timeregularity2d}
	Let the initial data $\bU_0$ and $\bW_0\in H^2(\dom)$ and $\Div\bU_0=0$. Then 
	\begin{equation*}
	\partial_{t}\bU,\partial_t\bW\in L^\infty(0,T;L^2(\dom))\cap L^2(0,T;H^1(\dom)).
	\end{equation*}
\end{lemma}
\begin{proof}
	In the Galerkin formulation~\eqref{eq:galerkinapprox}, we multiply by $\ajn'(t)$ and $\bjn'(t)$, add the resulting equations and sum over $j$:% and use~\eqref{eq:skew}	
	\begin{equation*}
	\begin{split}
	&\norm{\partial_{t}\bUN(t)}_{L^2}^2+\norm{\partial_{t}\bWN(t)}_{L^2}^2=(\nu+\nu_r)\int_{\dom}\Delta\bUN\partial_t\bUN+c_1\int_{\dom}\Delta\bWN\partial_t\bWN dx\\
	&\qquad +c_2\int_{\dom}\Grad\Div\bW\partial_t\bWN
	-\int_{\dom}(\bUN\cdot\Grad)\bUN\cdot \partial_{t}\bUN dx- \int_{\dom}(\bUN\cdot\Grad )\bWN \cdot \partial_{t}\bWN dx \\
	&\qquad\hphantom{}+ 2\nu_r\int_{\dom}\Curl\bWN \cdot \partial_{t}\bUN dx
	-4\nu_r\int_{\dom}\bWN\partial_t\bWN dx + 2\nu_r \int_{\dom}\Curl\bUN\cdot \partial_{t}\bWN dx.
	\end{split}
	\end{equation*}
	We estimate the right hand side at $t=0$:
	\begin{align*}
	&\norm{\partial_{t}\bUN(0)}_{L^2}^2+\norm{\partial_{t}\bWN(0)}_{L^2}^2\leq
	C\norm{\Delta\bUN^0}_{L^2}\norm{\partial_t\bUN(0)}_{L^2}+ C\norm{\Grad^2\bWN^0}_{L^2}\norm{\partial_t\bWN(0)}_{L^2}\\
	&\qquad +C\norm{\bUN^0}_{L^4}\norm{\Grad\bUN^0}_{L^4}\norm{\partial_t\bUN(0)}_{L^2} +C\norm{\bUN^0}_{L^4}\norm{\Grad\bWN^0}_{L^4}\norm{\partial_t\bWN(0)}_{L^2}\\
	&\qquad +C \norm{\Curl\bWN^0}_{L^2}\norm{\partial_t\bUN(0)}_{L^2}+C\norm{\bWN^0}_{L^2}\norm{\partial_t\bWN(0)}_{L^2}+C\norm{\Curl\bUN^0}_{L^2}\norm{\partial_t\bWN(0)}_{L^2}.
	\end{align*}
	Hence, using Ladyshenskaya's inequality, and $\norm{\bUN^0}_{H^2}\leq C\norm{\bU_0}_{H^2}$, $\norm{\bWN^0}_{H^2}\leq C\norm{\bW_0}_{H^2}$, we obtain
	\begin{equation}\label{eq:initmischt}
	\norm{\partial_{t}\bUN(0)}_{L^2}+\norm{\partial_{t}\bWN(0)}_{L^2}\leq
	C\norm{\bUN^0}_{H^2}^2 + C\norm{\bWN^0}_{H^2}^2\leq C(\norm{\bU_0}_{H^2}^2 +\norm{\bW_0}_{H^2}^2).
	\end{equation}
	Therefore $\partial_{t}\bUN(0),\partial_t\bWN(0)\in L^2(\dom)$ uniformly in $N$.	
	
	Next, we differentiate the equations~\eqref{eq:galerkinapprox} in time and use $\partial_t\bUN$ and $\partial_t\bWN$ respectively as test functions:
	\begin{align*}%\label{eq:galerkinapprox}
	&\frac{1}{2}\frac{d}{dt}\left(\norm{\partial_t\bUN}^2_{L^2}+\norm{\partial_t\bWN}^2_{L^2}\right)  +\nu\norm{\Grad\partial_t\bUN}^2_{L^2}+c_1\norm{\Grad\partial_t\bWN}_{L^2}^2+c_2\norm{\Div\partial_t\bWN}_{L^2}^2 \\
	&\qquad =2\nu_r\int_{\dom}\Curl\partial_t\bWN \cdot \partial_t\bUN dx-\nu_r\norm{\Curl\partial_t\bUN}^2_{L^2} -4\nu_r\norm{\partial_t\bWN}_{L^2}^2\\
	&\qquad \hphantom{=}-\int_{\dom}(\partial_t\bUN\cdot\Grad)\bUN\cdot \partial_t\bUN dx-\int_{\dom}(\bUN\cdot\Grad)\partial_t\bUN\cdot \partial_t\bUN dx\\
	& \qquad\hphantom{=}- \int_{\dom}(\partial_t\bUN\cdot\Grad )\bWN \cdot \partial_t\bWN dx- \int_{\dom}(\bUN\cdot\Grad )\partial_t\bWN \cdot \partial_t\bWN dx
	\\
	&\qquad\hphantom{=}+ 2\nu_r \int_{\dom}\Curl\partial_t\bUN\cdot\partial_t\bWN dx\\ 
	& \qquad =-\nu_r\norm{\Curl\partial_t\bUN-2\partial_t\bWN}_{L^2}^2 -\int_{\dom}(\partial_t\bUN\cdot\Grad)\bUN\cdot \partial_t\bUN dx\\
	& \qquad\hphantom{=}- \int_{\dom}(\partial_t\bUN\cdot\Grad )\bWN \cdot \partial_t\bWN dx.
	\end{align*}	
	The first term on the right hand side is nonpositive, so it remains to bound the other two terms. We use again Ladyshenskaya's and Young's inequality:
	\begin{align*}
	\left|\int_{\dom}(\partial_t\bUN\cdot\Grad)\bUN\cdot \partial_t\bUN dx\right|&\leq \norm{ \partial_t\bUN}_{L^4}^2\norm{\Grad\bUN}_{L^2}\\
	&\leq \norm{\Grad \partial_t\bUN}_{L^2}\norm{\partial_t\bUN}_{L^2}\norm{\Grad\bUN}_{L^2}\\
	&\leq \frac{\nu}{4}\norm{\Grad\partial_t\bUN}_{L^2}^2+\frac{C}{\nu}\norm{\partial_t\bUN}^2_{L^2}\norm{\Grad\bUN}_{L^2}^2,
	\end{align*}
	and similarly,
	\begin{align*}
	\left|\int_{\dom}(\partial_t\bUN\cdot\Grad)\bWN\cdot \partial_t\bWN dx\right|&\leq \norm{ \partial_t\bWN}_{L^4}\norm{\partial_t\bUN}_{L^4}\norm{\Grad\bWN}_{L^2}\\
	&\leq \norm{\Grad \partial_t\bUN}_{L^2}^{1/2}\norm{\partial_t\bUN}_{L^2}^{1/2}\norm{\Grad \partial_t\bWN}^{1/2}_{L^2}\norm{\partial_t\bWN}_{L^2}^{1/2}\norm{\Grad\bWN}_{L^2}\\
	&\leq C \norm{\Grad \partial_t\bUN}_{L^2}\norm{\partial_t\bUN}_{L^2}\norm{\Grad\bWN}_{L^2}\\
	&\qquad +C\norm{\Grad \partial_t\bWN}_{L^2}\norm{\partial_t\bWN}_{L^2}\norm{\Grad\bWN}_{L^2}\\
	&\leq \frac{\nu}{4}\norm{\Grad\partial_t\bUN}_{L^2}^2+\frac{C}{\nu}\norm{\partial_t\bUN}^2_{L^2}\norm{\Grad\bWN}_{L^2}^2\\
	&\qquad+\frac{c_1}{2}\norm{\Grad\partial_t\bWN}_{L^2}^2+\frac{C}{c_1}\norm{\partial_t\bWN}^2_{L^2}\norm{\Grad\bWN}_{L^2}^2.
	\end{align*}
	Therefore,
	\begin{align*}
	&\frac{1}{2}\frac{d}{dt}\left(\norm{\partial_t\bUN}^2_{L^2}+\norm{\partial_t\bWN}^2_{L^2}\right)  +\frac{\nu}{2}\norm{\Grad\partial_t\bUN}^2_{L^2}+\frac{c_1}{2}\norm{\Grad\partial_t\bWN}_{L^2}^2\\
	&\quad\leq \frac{C}{\nu}\norm{\partial_t\bUN}^2_{L^2}\left(\norm{\Grad\bUN}_{L^2}^2+\norm{\Grad\bWN}_{L^2}^2\right) + \frac{C}{c_1}\norm{\partial_t\bWN}^2_{L^2}\norm{\Grad\bWN}_{L^2}^2\\
	&\quad\leq C_{\nu,c_1}\left(\norm{\partial_t\bUN}^2_{L^2}+\norm{\partial_t\bWN}^2_{L^2}\right)\left(\norm{\Grad\bUN}_{L^2}^2+\norm{\Grad\bWN}_{L^2}^2\right) .
	\end{align*}	
	Since $\norm{\Grad\bUN}_{L^2}^2+\norm{\Grad\bWN}_{L^2}^2\in L^1([0,T])$, we can use Gr\"{o}nwall's inequality together with~\eqref{eq:initmischt} to conclude that
	\begin{equation*}
	\sup_{t\in[0,T]}\left(\norm{\partial_t\bUN}^2_{L^2}+\norm{\partial_t\bWN}^2_{L^2}\right) + \int_0^T\left(\norm{\Grad\partial_t\bUN}^2_{L^2}+\norm{\Grad\partial_t\bWN}_{L^2}^2\right) dt\leq C.
	\end{equation*}	
	This bound holds uniformly in $N$, and hence after passing $N\to \infty$ also for $\bU$ and $\bW$.	
\end{proof}	
Using this, we can prove the following lemma:
\begin{lemma}
	\label{lem:spaceregularity2d_1}
	Let the initial data $\bU_0$ and $\bW_0\in H^2(\dom)$ and $\Div\bU_0=0$ and assume that the domain $\dom$ is at least of class $C^2$. Then 
	\begin{equation*}
	\bU,\bW\in L^\infty(0,T;H^2(\dom)).
	\end{equation*}
\end{lemma}
\begin{proof}
	We write equations~\eqref{eq:simplified2d} in the variational form
	\begin{align}
	\label{eq:bilins}
	\begin{split}
	(\nu+\nu_r)\int_{\dom}\Grad\bU(t):\Grad\psi_1 dx & = \int_{\dom} g_1(t)\psi_1\, dx,\\
	\int_{\dom}\left(c_1\Grad\bW(t):\Grad\psi_2+ c_2\Div\bW(t)\Div\psi_2 \right) dx +4\nu_r\int_{\dom}\bW(t)\cdot\psi_2 dx & = \int_{\dom} g_2(t) \psi_2 dx,
	\end{split}
	\end{align}
	for any $\psi_1\in\V$, $\psi_2\in H^1_0(\dom)$, where $g_1, g_2$ are given by
	\begin{align*}
	g_1(t) &= 2\nu_r\Curl\bW(t) - \partial_t\bU(t)-(\bU(t)\cdot\Grad)\bU(t),\\
	g_2(t) &= 2\nu_r\Curl\bU(t) - \partial_t\bW(t)-(\bU(t)\cdot\Grad)\bW(t).
	\end{align*}
	The left hand sides in~\eqref{eq:bilins} are elliptic bilinear forms, and the right hand sides satisfy, using Lemma~\ref{lem:timeregularity2d} and the calculations in equations (3.102) in~\cite[Chapter III]{Temam2001},
	\begin{align*}
	&\partial_t\bU,\partial_t\bW\in L^\infty(0,T;L^2(\dom))\cap L^2(0,T;H^1(\dom)),\\
	&\Curl\bU,\Curl\bW\in L^\infty(0,T;L^2(\dom))\cap H^1(0,T;L^2(\dom)),\\
	& (\bU\cdot\Grad)\bU,(\bU\cdot\Grad)\bW\in L^\infty(0,T;L^{4/3}(\dom)),
	\end{align*}
	and hence $g_1,g_2\in L^\infty(0,T;L^{4/3}(\dom))$. Using Proposition I.2.2 in~\cite{Temam2001} for $\bU$ and elliptic regularity theory for $\bW$, we obtain that
	$\bU,\bW\in L^\infty(0,T;W^{2,4/3}(\dom))$. Hence by the Sobolev embedding, we obtain that $\bU,\bW\in L^\infty([0,T]\times\dom)$. As in the proof of Theorem III.3.6 in~\cite{Temam2001}, we can use this to obtain a better estimate on the convection terms and get $g_1,g_2\in L^\infty(0,T;L^2(\dom))$ and hence $\bU,\bW\in L^\infty(0,T;H^2(\dom))$. 
\end{proof}	
Higher order regularity follows from iterating this procedure and assuming that $\bU_0,\bW_0$ are regular enough:
\begin{lemma}\label{lem:higherorderregularity}
	Assume $\bU_0,\bW_0\in H^3(\dom)$ with $\Div\bU_0=0$ and that the domain $\dom$ is sufficiently smooth. Then
	\begin{equation*}
	\partial_t\bU,\partial_t\bW\in L^\infty(0,T;H^1(\dom)).
	\end{equation*}
\end{lemma}
\begin{proof}
	From Lemma~\ref{lem:timeregularity2d} and~\ref{lem:spaceregularity2d_1}, we have that
	\begin{equation*}
	\bU,\bW\in L^\infty(0,T;H^2(\dom)), \quad\partial_t\bU,\partial_{t}\bW\in L^\infty(0,T;L^2(\dom))\cap L^2(0,T;H^1(\dom)).
	\end{equation*}
	We consider again the Galerkin approximation~\eqref{eq:galerkinapprox}. We denote $\projU$ the projection onto the linear space spanned by the eigenfunctions $\{a_1,\dots,a_N\}$ and by $\projW$ the projection onto the space spanned by the eigenfunctions $\{d_1,\dots,d_N\}$. 
	The projections satisfy
	\begin{equation}\label{eq:projectionestimate}
	\begin{split}
	\norm{\projU v}_{L^2(\dom)}^2\leq \norm{v}_{L^2(\dom)}^2,&\quad \norm{\Grad\projU v}_{L^2(\dom)}^2\leq \norm{\Grad v}_{L^2(\dom)}^2\\
	\norm{\projW v}_{L^2(\dom)}^2\leq \norm{v}_{L^2(\dom)}^2,&\quad \norm{\Grad\projW v}_{L^2(\dom)}^2+\norm{\Div\projW v}_{L^2(\dom)}^2\leq \norm{\Grad v}_{L^2(\dom)}^2+\norm{\Div v}_{L^2(\dom)}^2.
	\end{split}
	\end{equation}
	Then the Galerkin formulation~\eqref{eq:galerkinapprox} can be written as
	\begin{equation*}
	\begin{split}
	\partial_t \bUN &= -\projU\left((\bUN\cdot\Grad)\bUN\right)+(\nu+\nu_r)\projU\Delta\bUN+2\nu_r\projU\Curl\bWN,\\
	\partial_t \bWN &= -\projW\left((\bUN\cdot\Grad)\bWN\right)+c_1\projW\Delta\bWN+c_2\projW\Grad\Div\bWN-4\nu_r\bWN +2\nu_r\projW\Curl\bUN.	
	\end{split}
	\end{equation*}
	and hence taking the gradient in both equations and estimating the $L^2$-norms at $t=0$, using~\eqref{eq:projectionestimate}, we obtain
	\begin{align*}
	\norm{\Grad\partial_t \bUN(0)}_{L^2(\dom)}&\leq \norm{\Grad\projU\left((\bUN^0\cdot\Grad)\bUN^0\right)}_{L^2}+(\nu+\nu_r)\norm{\Grad\projU\Delta\bUN^0}_{L^2}+2\nu_r\norm{\Grad\projU\Curl\bWN^0}_{L^2}\\
	&\leq \norm{\Grad\left((\bUN^0\cdot\Grad)\bUN^0\right)}_{L^2}+(\nu+\nu_r)\norm{\Grad\Delta\bUN^0}_{L^2}+2\nu_r\norm{\Grad\Curl\bWN^0}_{L^2}\\
	&\leq \norm{\Grad\bUN^0}^2_{L^4}+\norm{\bUN^0}_{L^\infty}\norm{\Grad^2\bUN}_{L^2} +(\nu+\nu_r)\norm{\Grad^3\bUN^0}_{L^2}+2\nu_r\norm{\Grad^2\bWN^0}_{L^2}\\
	&\leq C \left(\norm{\bUN^0}_{H^3}+\norm{\bWN^0}_{H^2}\right)\leq C\left(\norm{\bU_0}_{H^3}+\norm{\bW_0}_{H^2}\right),
	\end{align*}
	and similarly,	
	\begin{align*}
	\norm{\Grad\partial_t \bWN(0)}_{L^2(\dom)}&\leq \norm{\Grad\projW\left((\bUN^0\cdot\Grad)\bWN^0\right)}_{L^2}+c_1\norm{\Grad\projW\Delta\bWN^0}_{L^2}+c_2\norm{\Grad\projW\Grad\Div\bWN^0}\\
	&\quad+2\nu_r\norm{\Grad\projW\Curl\bUN^0}_{L^2}+4\nu_r\norm{\Grad\bWN^0}_{L^2}\\
	&\leq
	C\norm{\Grad\left((\bUN^0\cdot\Grad)\bWN^0\right)}_{L^2}+C\norm{\Grad\Delta\bWN^0}_{L^2}+C\norm{\Grad\Grad\Div\bWN^0}_{L^2}\\
	&\quad+C\norm{\Grad\Curl\bUN^0}_{L^2}+4\nu_r\norm{\Grad\bWN^0}_{L^2}\\
	&\leq C\norm{\Grad\bUN^0}_{L^4}\norm{\Grad\bWN^0}_{L^4} + C\norm{\bUN^0}_{L^\infty}\norm{\Grad^2\bWN^0}_{L^2}+C\norm{\Grad^3\bWN^0}_{L^2}\\
	&\quad+C\norm{\Grad^2\bUN^0}_{L^2}+4\nu_r\norm{\Grad\bWN^0}_{L^2}\\
	&\leq C\left(\norm{\bU_0}_{H^2}+\norm{\bW_0}_{H^3}\right),
	\end{align*}
	and hence
	\begin{equation*}
	\norm{\Grad\partial_t\bUN(0)}_{L^2}+\norm{\Grad\partial_t\bWN(0)}_{L^2}\leq C\left(\norm{\bU_0}_{H^3}+\norm{\bW_0}_{H^3}\right).
	\end{equation*}
	Next we multiply the Galerkin formulations~\eqref{eq:galerkinapprox} by the eigenvalues $\lambda_j$ and $\xi_j$ respectively and differentiate in time:
	\begin{equation*}
	\begin{split}
	&\frac{d}{dt}\int_{\dom}\partial_t\bUN\cdot \Delta a_j dx +\int_{\dom}(\partial_t\bUN\cdot\Grad)\bUN\cdot \Delta a_j dx+\int_{\dom}(\bUN\cdot\Grad)\partial_t\bUN\cdot \Delta a_j dx\\
	&\quad =(\nu+\nu_r)\int_{\dom}\Delta\partial_t\bUN \cdot \Delta a_j dx 
	+ 2\nu_r\int_{\dom}\Curl\partial_t\bWN \cdot \Delta a_j dx,\\
	&\frac{d}{dt}\int_{\dom}\partial_t\bWN\cdot (c_1\Delta +c_2\Grad\Div)d_j dx + \int_{\dom}(\partial_t\bUN\cdot\Grad )\bWN \cdot (c_1\Delta +c_2\Grad\Div)d_j dx \\
	&\qquad + \int_{\dom}(\bUN\cdot\Grad )\partial_t\bWN \cdot (c_1\Delta +c_2\Grad\Div) d_j dx\\
	&\quad =c_1\int_{\dom}\Delta \partial_t \bWN\cdot (c_1\Delta +c_2\Grad\Div) d_j dx + c_2\int_{\dom}\Grad\Div\partial_t\bWN \cdot (c_1\Delta +c_2\Grad\Div) d_j dx  \\
	& \qquad-4\nu_r\int_{\dom}\partial_t\bWN\cdot (c_1\Delta +c_2\Grad\Div) d_ j dx+ 2\nu_r \int_{\dom}\Curl\partial_t\bUN\cdot (c_1\Delta +c_2\Grad\Div) d_j dx.
	\end{split}
	\end{equation*}
	Then we multiply by $\alpha_{j,N}'$ and $\beta_{j,N}'$ respectively and sum over $j=1,\dots, N$,
	\begin{equation*}
	\begin{split}
	&\frac{d}{dt}\int_{\dom}|\partial_t\Grad\bUN|^2 dx -\int_{\dom}(\partial_t\bUN\cdot\Grad)\bUN\cdot \Delta \partial_t\bUN dx-\int_{\dom}(\bUN\cdot\Grad)\partial_t\bUN\cdot \Delta \partial_t\bUN dx\\
	&\quad =-(\nu+\nu_r)\int_{\dom}|\Delta\partial_t\bUN|^2 dx 
	- 2\nu_r\int_{\dom}\Curl\partial_t\bWN \cdot \Delta \partial_t\bUN dx,\\
	&\frac{d}{dt}\int_{\dom}\left(c_1|\partial_t\Grad\bWN|^2+c_2|\partial_t\Div\bWN|^2\right) dx - \int_{\dom}(\partial_t\bUN\cdot\Grad )\bWN \cdot (c_1\Delta +c_2\Grad\Div) \partial_t\bWN  dx \\
	&\qquad - \int_{\dom}(\bUN\cdot\Grad )\partial_t\bWN \cdot (c_1\Delta +c_2\Grad\Div) \partial_t\bWN  dx+\int_{\dom}| (c_1\Delta +c_2\Grad\Div) \partial_t\bWN|^2  dx\\
	&\quad =-4\nu_r\int_{\dom}\left(c_1|\partial_t\Grad\bWN|^2+c_2|\Div\partial_t\bWN|^2\right) dx- 2\nu_r \int_{\dom}\Curl\partial_t\bUN\cdot (c_1\Delta +c_2\Grad\Div) \partial_t\bWN  dx.
	\end{split}
	\end{equation*}
	Adding the two equations and rearranging, we get
	\begin{equation}\label{eq:faen3}
	\begin{split}
	&\frac{1}{2}\frac{d}{dt}\left(\norm{\Grad\partial_{t}\bUN}_{L^2}^2+c_1\norm{\Grad\partial_{t}\bWN}_{L^2}^2+c_2\norm{\Div\partial_t\bWN}_{L^2}^2\right) +(\nu+\nu_r)\norm{\Delta\partial_t\bUN}_{L^2}^2 \\
	&\qquad\hphantom{=}  +\norm{(c_1\Delta +c_2\Grad\Div) \partial_t\bWN}_{L^2}^2 +4c_1\nu_r\norm{\Grad\partial_t\bWN}_{L^2}^2+4c_2\nu_r\norm{\Div\partial_t\bWN}_{L^2}^2\\
	&\qquad =  \int_{\dom}[(\partial_t\bUN\cdot\Grad)\bUN+(\bUN\cdot\Grad)\partial_t\bUN]\cdot\Delta\partial_{t}\bUN dx\\
	&\qquad \hphantom{=}
	+\int_{\dom}[(\partial_t\bUN\cdot\Grad )\bWN+(\bUN\cdot\Grad )\partial_t\bWN]\cdot (c_1\Delta +c_2\Grad\Div) \partial_t\bWN dx\\
	&\qquad \hphantom{=}- 2\nu_r \int_{\dom}\Curl\partial_t\bUN\cdot (c_1\Delta +c_2\Grad\Div) \partial_t\bWN  dx	- 2\nu_r\int_{\dom}\Curl\partial_t\bWN \cdot \Delta \partial_t\bUN dx. 
	\end{split}
	\end{equation}
	The terms on the right hand side can be estimated as follows (using Ladyshenskaya's inequality):
	\begin{align*}
	&\left|\int_{\dom}[(\partial_t\bUN\cdot\Grad )\bWN]\cdot(c_1\Delta+c_2\Grad\Div) \partial_t\bWN dx\right|\leq \norm{\partial_t\bUN}_{L^4}\norm{\Grad\bWN}_{L^4}\norm{(c_1\Delta+c_2\Grad\Div)\bWN}_{L^2}\\
	&\qquad\leq
	\frac{1}{4} \norm{(c_1\Delta+c_2\Grad\Div)\bWN}_{L^2}^2 + C \norm{\partial_t\bUN}_{L^4}^2\norm{\Grad\bWN}_{L^4}^2\\
	&\qquad\leq
	\frac{1}{4} \norm{(c_1\Delta+c_2\Grad\Div)\bWN}_{L^2}^2 + C \norm{\partial_t\bUN}_{L^2}\norm{\partial_t\Grad\bUN}_{L^2}\norm{\Grad\bWN}_{L^2}\norm{\Grad^2\bWN}_{L^2}\\
	&\qquad\leq
	\frac{1}{4} \norm{(c_1\Delta+c_2\Grad\Div)\bWN}_{L^2}^2 + C \norm{\partial_t\bUN}_{L^2}^2+ C\norm{\partial_t\Grad\bUN}_{L^2}^2\norm{\Grad\bWN}_{L^2}^2\norm{\Grad^2\bWN}_{L^2}^2.
	\end{align*}
	Similarly,
	\begin{align*}
	&\left|\int_{\dom}[(\bUN\cdot\Grad )\partial_t\bWN]\cdot(c_1\Delta+c_2\Grad\Div) \partial_t\bWN dx\right|\leq \norm{\bUN}_{L^\infty}\norm{\partial_t\Grad\bWN}_{L^2}\norm{(c_1\Delta+c_2\Grad\Div)\bWN}_{L^2}\\
	&\qquad\leq
	\frac{1}{4} \norm{(c_1\Delta+c_2\Grad\Div)\bWN}_{L^2}^2 + C \norm{\bUN}_{L^\infty}^2\norm{\Grad\partial_t\bWN}_{L^2}^2.
	\end{align*}
	and replacing $\bWN$ by $\bUN$, also
	\begin{align*}
	&\left|\int_{\dom}[(\partial_t\bUN\cdot\Grad )\bUN]\cdot\Delta \partial_t\bUN dx\right|\leq \norm{\partial_t\bUN}_{L^4}\norm{\Grad\bUN}_{L^4}\norm{\Delta\partial_t\bUN}_{L^2}\\
	&\qquad\leq
	\frac{\nu}{4}\norm{\Delta\partial_t\bUN}_{L^2}^2+\frac{C}{\nu}\norm{\partial_t\bUN}_{L^4}^2\norm{\Grad\bUN}_{L^4}^2\\
	&\qquad\leq
	\frac{\nu}{4}\norm{\Delta\partial_t\bUN}_{L^2}^2+\frac{C}{\nu}\norm{\partial_t\bUN}_{L^2}\norm{\partial_t\Grad\bUN}_{L^2}\norm{\Grad\bUN}_{L^2}\norm{\Grad^2\bUN}_{L^2}\\
	&\qquad\leq
	\frac{\nu}{4}\norm{\Delta\partial_t\bUN}_{L^2}^2+\frac{C}{\nu}\norm{\partial_t\bUN}_{L^2}^2+\frac{C}{\nu}\norm{\partial_t\Grad\bUN}_{L^2}^2\norm{\Grad\bUN}_{L^2}^2\norm{\Grad^2\bUN}_{L^2}^2.
	\end{align*}
	and
	\begin{align*}
	&\left|\int_{\dom}[(\bUN\cdot\Grad )\partial_t\bUN]\cdot\Delta \partial_t\bUN dx\right|\leq \norm{\bUN}_{L^\infty}\norm{\Delta\partial_t\bUN}_{L^2}\norm{\Grad\partial_t\bUN}_{L^2}\\
	&\qquad\leq \frac{\nu}{4}\norm{\Delta\partial_t\bUN}_{L^2}^2+\frac{C}{\nu}\norm{\Grad\partial_t\bUN}_{L^2}^{2}\norm{\bUN}_{L^\infty}^2.
	\end{align*}
	The remaining two terms can be estimated as
	\begin{equation*}
	2\nu_r\left| \int_{\dom}\Curl\partial_t\bUN\cdot (c_1\Delta +c_2\Grad\Div) \partial_t\bWN  dx \right|\leq \frac{1}{4}\norm{(c_1\Delta +c_2\Grad\Div) \partial_t\bWN}_{L^2}^2 +C\nu_r^2\norm{\Grad\partial_t\bUN}_{L^2}^2
	\end{equation*}
	and
	\begin{equation*}
	2\nu_r\left|\int_{\dom}\Curl\partial_t\bWN \cdot \Delta \partial_t\bUN dx\right|\leq\frac{\nu}{4}\norm{\Delta\partial_t\bUN}_{L^2}^2+\frac{C\nu_r^2}{\nu}\norm{\Grad\partial_t\bWN}_{L^2}^2
	\end{equation*}	
	Hence, we can upper bound in equation~\eqref{eq:faen3}
	\begin{equation}\label{eq:faen4}
	\begin{split}
	&\frac{1}{2}\frac{d}{dt}\left(\norm{\Grad\partial_{t}\bUN}_{L^2}^2+\norm{\Grad\partial_{t}\bWN}_{L^2}^2\right) +\frac{\nu}{4}\norm{\Delta\partial_t\bUN}_{L^2}^2+\frac{1}{4}\norm{(c_1\Delta+c_2\Grad\Div)\partial_t\bWN}_{L^2}^2 \\
	&\qquad\leq
	C\norm{\partial_t\Grad\bUN}_{L^2}^2\norm{\Grad\bWN}_{L^2}^2\norm{\Grad^2\bWN}_{L^2}^2\\
	&\qquad \hphantom{\leq} + C \norm{\partial_t\bUN}_{L^2}^2+ C\norm{\partial_t\Grad\bUN}_{L^2}^2\norm{\Grad\bWN}_{L^2}^2\norm{\Grad^2\bWN}_{L^2}^2\\
	&\qquad\hphantom{\leq}+  C \norm{\bUN}_{L^\infty}^2\norm{\Grad\partial_t\bWN}_{L^2}^2 +\frac{C}{\nu}\norm{\partial_t\bUN}_{L^2}^2+\frac{C}{\nu}\norm{\partial_t\Grad\bUN}_{L^2}^2\norm{\Grad\bUN}_{L^2}^2\norm{\Grad^2\bUN}_{L^2}^2\\
	&\qquad\hphantom{\leq}+\frac{C}{\nu}\norm{\Grad\partial_t\bUN}_{L^2}^{2}\norm{\bUN}_{L^\infty}^2+C\nu_r^2\norm{\Grad\partial_t\bUN}_{L^2}^2+\frac{C\nu_r^2}{\nu}\norm{\Grad\partial_t\bWN}_{L^2}^2.
	\end{split}
	\end{equation}
	Thanks to Lemmas~\ref{lem:timeregularity2d} and~\ref{lem:spaceregularity2d_1}, the right hand side can be bounded uniformly in $t\in [0,T]$ by
	\begin{equation*}
	C_{\nu,\nu_r,c_1} \left(\norm{\Grad\partial_t\bUN}_{L^2}^{2}+\norm{\Grad\partial_t\bWN}_{L^2}^{2}\right) +C_{\nu,\nu_r,c_1},
	\end{equation*}
	and an application of the Gr\"{o}nwall inequality and letting $N\to\infty$ yields the result.	
\end{proof}	
\begin{lemma}
	\label{lem:higherregularityspace2d}
	Assume $\bU_0,\bW_0\in H^3(\dom)$ with $\Div\bU_0=0$. Then
	\begin{equation*}
	\bU,\bW\in  L^\infty(0,T;H^3(\dom)).	
	\end{equation*}	
\end{lemma}
\begin{proof}
	We proceed as in Lemma~\ref{lem:spaceregularity2d_1} and write the equations in the form~\eqref{eq:bilins}. From Lemma~\ref{lem:spaceregularity2d_1}, we already get that
	\begin{equation*}
	\Curl\bW,\Curl\bU,(\bU\cdot\Grad)\bU,(\bU\cdot\Grad)\bW\in L^\infty(0,T;H^1(\dom)). 
	\end{equation*}
	Lemma~\ref{lem:higherorderregularity} additionally yields that 
	\begin{equation*}
	\partial_t\bU,\partial_t\bW\in L^\infty(0,T;H^1(\dom)).
	\end{equation*}
	Hence $g_1(t),g_2(t)\in H^1(\dom)$ for almost every $t$. Thus Proposition I.2.2 in~\cite{Temam2001} and elliptic regularity imply that 
	\begin{equation*}
	\bU,\bW\in L^\infty(0,T;H^3(\dom)).
	\end{equation*}	
\end{proof}
\begin{remark}
	One could iterate this further to achieve even more regularity of $\bU$ and $\bW$ under the assumption that the domain and the initial data are smooth enough.
\end{remark}
\begin{remark}
	\label{rem:lipschitz}
	Notice that this implies that $\bU,\bW\in L^\infty(0,T;\mathrm{Lip}(\dom))$. In addition, $\bU$, $\bW$, $\bH=\Grad\Phi$ satisfy an energy balance with equality:
	\begin{multline*}
	\frac{1}{2}\frac{d}{dt}\int_{\dom}\left(|\bU|^2+|\bW|^2 + \mu_0(1+\kappa_0)|\bH|^2\right) dx\\ +\int_{\dom}\left(\nu|\Grad\bU|^2 +c_1|\Grad\bW|^2 + c_2|\Div\bW|^2 + \nu_r|\Curl\bU-2\bW|^2\right)dx= \mu_0\int_{\dom}\partial_t\bh_a\bH dx.
	\end{multline*}
\end{remark}
In a very similar way, one can prove uniqueness and regularity of solutions locally in time for the three dimensional case (see also~\cite[Theorem III.3.11]{Temam2001} and \cite{Amirat2010}). We state the result here without proof:
\begin{theorem}
	\label{thm:3dregularity}
	Let $d=3$ and assume $\bU_0,\bW_0\in H^3(\dom)$ with $\Div\bU_0=0$ and that the domain $\dom$ and the boundary data $\bh_a$ are sufficiently smooth. Then there exists $T^*>0$ such that the solution of~\eqref{eq:limitsys} satisfies
	\begin{equation*}
	\partial_t\bU,\partial_t\bW,\partial_t\bH\in L^\infty(0,T^*;H^1(\dom))\cap L^2(0,T^*;H^2(\dom))\quad \text{and }\quad \bU,\bW,\bH\in L^{\infty}(0,T^*;H^3(\dom)).
	\end{equation*}

\end{theorem}  
\section{Adaptions needed to prove existence of weak solutions in~\cite{Amirat2008} for chosen boundary conditions}\label{a:modifications}
The goal of this section is to explain the adaptions that have to be made in the proof of existence of weak solutions for~\eqref{eq:approxsys} to accommodate the boundary conditions
\begin{equation}
\label{eq:hbc}
\bhs\cdot\normal = (\bh_a-\bms)\cdot \normal, \quad \text{on } [0,T]\times \partial\dom, 
\end{equation}
for the magnetizing field $\bhs$ and the natural boundary conditions
\begin{equation}
\label{eq:mbc}
\Div\bms =0,\quad \Curl\bms\times\normal=0,\quad \text{on } [0,T]\times \partial\dom, 
\end{equation}
for the magnetization $\bms$ instead of 
\begin{equation*}
\bhs\cdot\normal =0,\quad  \text{on } [0,T]\times \partial\dom, 
\end{equation*}
and
\begin{equation}\label{eq:Muratbc}
\bms\cdot\normal=0,\quad \Curl\bms\times\normal =0,  \quad\text{on } [0,T]\times \partial\dom, 
\end{equation}
respectively, as in~\cite{Amirat2008}. With boundary conditions~\eqref{eq:mbc}, the solution $\bms(t)$ is sought in the space $\mathcal{K}$ as defined in~\eqref{eq:spaceform} instead of 
\begin{equation*}
\mathcal{K}_0=\{v\in L^2(\dom)\,|\,\Div v,\Curl v\in L^2(\dom), \,v\cdot\normal =0,\text{ on } \partial\dom \}=\{v\in H^1(\dom)\,|\, v\cdot\normal =0\text{ on } \partial\dom \},
\end{equation*}
as in~\cite{Amirat2008}. The reason why we need to do this, is that in the Galerkin approximation for the existence proof in~\cite[pp. 336--343]{Amirat2008}, one would like to take the approximation of $\bhs$ as a test function in the weak formulation for the approximation of $\bms$ to obtain an energy inequality. With boundary conditions~\eqref{eq:hbc}, the approximation of $\bhs$ is not in the space $\mathcal{K}_0$ and hence not a valid test function. By changing the function space for $\bms$ to $\mathcal{K}$, $\bhs$ and its approximations become a valid test function.

$\mathcal{K}$ equipped with inner product
\begin{equation*}
\langle q_1, q_2\rangle := \int_{\dom} q_1\cdot q_2\, dx +\int_{\dom} \Div q_1 \Div q_2\, dx + \int_{\dom}\Curl q_1\cdot \Curl q_2\, dx.
\end{equation*}
is a Hilbert space and has a subspace
\begin{equation*}
\mathcal{N}:=\overline{\text{span}\{v\in H^1(\dom)\,|\, v=\Grad\psi,\text{ for some } \psi\in H^2(\dom) \}}^{\mathcal{K}}\subset \mathcal{K}.
\end{equation*}
Let $\mathcal{N}^\perp$ be the orthogonal complement of this subspace. We can find a smooth orthonormal basis $\{\Grad \eta_j \}_{j=1}^\infty$ of $\mathcal{N}$ and a smooth orthonormal basis $\{\theta_j\}_{j=1}^\infty$ of its complement $\mathcal{K}^\perp$ such that $\{\Grad \eta_j \}_{j=1}^\infty\cup \{\theta_j\}_{j=1}^\infty$ is an orthonormal basis of $\mathcal{K}$.
The boundary conditions~\eqref{eq:mbc} are natural and incorporated in the weak formulation for the magnetization $\bms$ instead of the function space as in the case of the boundary condition $\bms\cdot\normal =0$. Indeed, for a smooth vector field $v$ satisfying~\eqref{eq:mbc} and a test function $q\in \mathcal{K}$, it holds
\begin{equation*}
\begin{split}
-\int_{\dom}\Delta v \cdot q dx &=\int_{\dom}\Curl v\cdot \Curl q dx +\int_{\dom}\Div v \Div q dx -\int_{\partial\dom}(\Curl v\times\normal)\cdot q dS -\int_{\partial\dom}\Div v q\cdot \normal dS \\
&=\int_{\dom}\Curl v\cdot \Curl q dx +\int_{\dom}\Div v \Div q dx,
\end{split}
\end{equation*}
and hence the correct weak formulation of~\eqref{seq:dtma} is~\eqref{eq:weakapproxm}, the same as for boundary conditions~\eqref{eq:Muratbc} except that in the case of~\eqref{eq:Muratbc}, the test functions should be taken in the space $\mathcal{K}_0$. 

In the Galerkin approximation for the existence proof (\cite[p. 337]{Amirat2008}), one then instead defines $\bms_n$ and $\bhs_n$, the approximations at level $n$, by
\begin{equation*}
\bms_n =\sum_{j=1}^n\gamma_j^n(t)\theta_j +\sum_{j=1}^n\gamma^n_{j+n}(t)\Grad\eta_j,\quad \bhs_n=\Grad\phis_n,\quad\text{with } \phis_n=\sum_{j=1}^n\delta_j^n(t)\eta_j,
\end{equation*}
where $\eta_j$ and $\theta_j$ are as in the orthonormal bases of $\mathcal{N}$ and $\mathcal{N}^\perp$ and $\gamma_j^n$ and $\delta_j^n$ are defined from the conditions ($j=1,\dots,n$)

\begin{align*}
&\frac{d}{dt}\int_{\dom}\bms_n\cdot\Grad\eta_j \,dx +\int_{\dom}[(\bus_n\cdot\Grad)\bms_n]\cdot\Grad\eta_j\, dx + \sigma\int_{\dom}\Div\bms_n\Delta\eta_j\, dx\\
&\qquad\qquad\qquad= \int_{\dom}(\bws_n\times\bms_n)\cdot\Grad\eta_j\, dx-\frac{1}{\tau}\int_{\dom}(\bms_n-\kappa_0\bhs_n)\cdot\Grad\eta_j\, dx,\\
&\frac{d}{dt}\int_{\dom}\bms_n\cdot\theta_j \,dx +\int_{\dom}[(\bus_n\cdot\Grad)\bms_n]\cdot\theta_j\, dx +
\sigma\int_{\dom}\Curl\bms_n:\Curl\theta_j\, dx + \sigma\int_{\dom}\Div\bms_n\Div\theta_j\, dx\notag\\
&\qquad\qquad\qquad= \int_{\dom}(\bws_n\times\bms_n)\cdot\theta_j\, dx-\frac{1}{\tau}\int_{\dom}(\bms_n-\kappa_0\bhs_n)\cdot\theta_j\, dx,\\
&\qquad \qquad \bms_n(0,\cdot)=\bm_{0,n};\\
&\int_{\dom}\Grad\phis_n\cdot\Grad\eta_j\, dx =-\int_{\dom}(\bms_n-\bh_a)\Grad\eta_j\, dx; 
\end{align*}
where $\bm_{0,n}$ is the orthogonal projection of $\bm_0$ in $L^2(\dom)$ onto the space spanned by $\{\Grad\eta_j \}_{j=1}^n\cup\{\theta_j\}_{j=1}^n$.

Then one proceeds as in the existence proof in~\cite{Amirat2008}. The only other difference is that since $\{\bms_n\}_n\subset \mathcal{K}$ instead of $H^1(\dom)$, one needs to use the Div-Curl Lemma (see e.g. \cite[Theorem 5.2.1]{Evans1990}) to obtain that the space $\mathcal{K}$ is compact in $L^2(\dom)$ (instead of Rellich's theorem as for the embedding $H^1(\dom)\subset\subset L^2(\dom)$) and then apply the Aubin-Lions Lemma with this space to obtain strong convergence in $L^2$ of a subsequence of $\{\bms_n\}_n$ to some limiting function $\bms$.

\smallskip

\bibliographystyle{abbrv}
\bibliography{Rosensweigbib}
\end{document}